\documentclass{cmslatex}
\usepackage[paperwidth=7in, paperheight=10in, margin=.875in]{geometry}
 \usepackage[backref,colorlinks,linkcolor=red,anchorcolor=green,citecolor=blue]{hyperref}
\usepackage{amsfonts,amssymb}
\usepackage{amsmath}
\usepackage{graphicx}
\usepackage{cite}
\usepackage{enumerate}
\usepackage{mathabx}
\usepackage[utf8]{inputenc} 
\usepackage[english]{babel}
\usepackage{amstext}
\usepackage{euscript}
\usepackage{bbm}
\usepackage{mathrsfs}

\usepackage{amstext}
\usepackage{amsfonts}
\usepackage{latexsym}
\usepackage{mathtools} 
\usepackage{enumitem} 
\usepackage{accents}

\usepackage{bm}

\renewcommand{\theequation}{\arabic{section}.\arabic{equation}}

\newcommand*{\im}{\mathop{}\!\mathrm{i}}
\newcommand*{\e}{\mathop{}\!\mathrm{e}}

\newcommand{\rd}{\mathrm{d}}

\newcommand{\mP}{\mathcal{P}}

\newcommand{\bR}{\mathbb{R}}

\newcommand{\Sd}{\mathbb{S}}
\newcommand{\D}{\mathcal{D}_L}

\newcommand{\fe}{f_{N}}
\newcommand{\gn}{g_{N}}

\newcommand{\n}{\bm{\mathrm{n}}}
\newcommand{\kk}{\bm{\mathrm{k}}}
\newcommand{\lb}{\bm{\mathrm{l}}}
\newcommand{\m}{\bm{\mathrm{m}}}
\newcommand{\ib}{\bm{\mathrm{i}}}
\newcommand{\jb}{\bm{\mathrm{j}}}

   \allowdisplaybreaks
\begin{document}
 \title{Convergence of the Fourier-Galerkin spectral method for the Boltzmann equation with uncertainties\thanks{Received date, and accepted date (The correct dates will be entered by the editor).}}


          \author{Liu Liu\thanks{Department of Mathematics, The Chinese University of Hong Kong, Shatin, Hong Kong, People's Republic of China, (lliu@math.cuhk.edu.hk).}
          \and Kunlun Qi \thanks{{School of Mathematics, University of Minnesota--Twin Cities, Minneapolis, MN 55455 USA, (kqi@umn.edu).}}}

         \pagestyle{myheadings} \markboth{Boltzmann equation with uncertainties}{L.~Liu and K.~Qi} 
         \maketitle

          \begin{abstract}
               It is well-known that the Fourier-Galerkin spectral method has been a popular approach for the numerical approximation of the deterministic Boltzmann equation with spectral accuracy rigorously proved. In this paper, we will show that such a spectral convergence of the Fourier-Galerkin spectral method also holds for the Boltzmann equation with uncertainties arising from both collision kernel and initial condition. Our proof is based on newly-established spaces and norms that are carefully designed and take the velocity variable and random variables with their high regularities into account altogether. For future studies, this theoretical result will provide a solid foundation for further showing the convergence of the full-discretized system where both the velocity and random variables are discretized simultaneously.
          \end{abstract}

          
\section{Introduction}
\label{sec:intro}

\textbf{Background and goals} Kinetic equations have been widely used in vast important areas including rarefied gas, plasma physics, astrophysics and new realms such as semiconductor device modeling \cite{Markowich}, environmental, social and biological sciences \cite{BGK-book}. They describe the non-equilibrium dynamics of a gas or system composed of a large number of particles.
The Boltzmann equation, as a typical example, is used to model different phenomena ranging from rarefied gas flows found in hypersonic aerodynamics, gases in vacuum technologies, or fluids inside microelectromechanical devices \cite{cercignani}, to the description of social and biological phenomena \cite{BGK-book, PT2}. We refer to \cite{DPR,pareschi2001introduction,rjasanow} for recent monographs, collections and surveys.

Due to the significant role the Boltzmann equation plays in the multi-scale hierarchy, in the last decades, the research on kinetic theory has attracted attention not only from the theoretical perspective but also in the numerical field. Derived from $N$-body Newton's equations in a certain sense of limit \cite{cercignani}, the Boltzmann equation typically contains a highly-dimensional and nonlinear integral operator modeling interactions between particles, which brings huge challenges in its numerical approximation. Though the stochastic type method such as the direct simulation Monte Carlo (DSMC) method \cite{Nanbu80,Bird} showed its benefits in efficiency and simplicity of implementation, the deterministic type of methods have developed swiftly in recent decades thanks to the improvement of the computing powers \cite{DP15}. Particularly, to approximate the Boltzmann collision operator the Fourier-Galerkin spectral method provides us a suitable framework \cite{PP96,PR05}, which not only enjoys the spectral accuracy but also can be easily adapted by some accelerated algorithms \cite{BR99, MP06, GHHH17}. In addition to the success of the Fourier-Galerkin spectral method in numerical simulation, the theoretical proof of its spectral convergence is yet to be complete over a long period of time, unless some filters are added to keep the positivity of solution \cite{PR00stability}. 
Recently, motivated by the pioneering work \cite{FM11} of the stability analysis that relies on the "spreading" property of the collision operator, the authors in \cite{HQY21} presented a new framework to show the stability and convergence of the Fourier-Galerkin spectral method, based on a careful $L_v^2$ estimate of the negative part of the solution to the deterministic homogeneous Boltzmann equation, which also enables us to show the convergence when applying the Fourier-Galerkin spectral method to other variant types of the Boltzmann models. 

On the other hand, studying the corresponding uncertainty quantification (UQ) problems is important to assess, validate and improve the underlying models, which is necessary to obtain more reliable predictions and risk assessment. In particular, the collision kernel or scattering cross-section in the Boltzmann integral operator describes the transition rate during particle collisions. Calculating the collision kernel from first principles is extremely complicated and almost impossible for complex particle systems, thus only heuristic approximations are available and it is inevitable that the collision kernel contains uncertainties. Besides, inaccurate measurements of the initial or boundary data, forcing and source terms may bring other sources of uncertainties. Despite the numerous existing research on the Boltzmann and related equations, the study of kinetic models with random uncertainties has only started in recent years \cite{HPY, Jin-ICM, PZ2020, Poette2, RHS, HJ16, LW2017, CMT18}. We refer to the recent collection \cite{JinPareschi}, survey \cite{parUQ} and some relevant works \cite{PIN_book,LeMK,Bio-UQ,NTW,PDL,Schwab}. In particular, we mention a relevant work: in \cite{LJ18} and \cite{DJL19}, the authors provided powerful theoretical tools based on hypocoercivity to study and conduct the local sensitivity analysis for a class of multiscale, non-homogeneous kinetic equations with random uncertainties and their numerical approximations by using the gPC-SG method in random variable discretization. 

Given the motivations above, in this paper, our main purpose is to show that stability and convergence of the Fourier-Galerkin spectral method still hold for the Boltzmann equation with uncertainties that arise from both initial data and collision kernel. 

\textbf{Challenges and main contributions} In contrast with deterministic problems, where the parameters in the collision operator and initial conditions are completely certain, we have to deal with the impact of various sources of random inputs brought into the system. In this paper, we concentrate on the discretization of the numerical system within the deterministic variable domain, for which the Fourier-Galerkin spectral approximation is specifically employed. Notably, our investigation involves a detailed examination of the regularity of the numerical solution in the random space. This analysis serves as a crucial theoretical foundation, facilitating the demonstration of convergent results for the semi-discretized system in both velocity and random variables in our follow-up work.

Hence, considering the challenges caused by the random variable, the usual function space defined only in the velocity variable is far from sufficient. One essential ingredient in our analysis is to develop a new type of Sobolev space that involves the random variable and its higher-order derivatives, based on which the stability and convergence of the solution are analyzed. 

\textbf{Structure of our paper} The rest of this paper is organized as follows. In section \ref{sec:Boltzmann}, we mainly introduce the Boltzmann equation with uncertainties and formulation of the associated numerical system for velocity and random variables. As a preliminary preparation, the newly-established spaces and norms that take into account the uncertainties are presented in section \ref{sec:pre}, where desired assumptions on the collision kernel and initial datum, in addition to estimation of the collision operator with uncertainties, will be clarified. 
In section \ref{sec:main}, by first studying the propagation of the numerical solution in designated norms, we extend the local well-posedness of the numerical solution to any arbitrarily selected time interval. The convergence of the Fourier-Galerkin system with uncertainties is shown in section \ref{sec:conv} and conclusions are finally drawn in section \ref{sec:con}.


\section{The Boltzmann equation and associated numerical system}
\label{sec:Boltzmann}

\subsection{The Boltzmann equation with uncertainties}
\label{subsec:Boltzmann_uncertain}

We consider the spatially homogeneous Boltzmann equation with random inputs, 
\begin{equation}\label{IBE}
\partial_{t} f (t,v,z) = Q(f,f)(t,v,z), \quad t>0, \ v\in \mathbb{R}^d, \, z \in I_{z},
\end{equation}
where $d\geq 2$. The initial datum is assumed
\begin{equation}\label{f0}
f(0,v,z) = f^{0}(v,z), \quad v\in \mathbb{R}^d, \ d\geq 2,\ z \in I_{z},
\end{equation}
where $f=f(t,v,z)$ is the probability density function at time $t$, with velocity $v$ and an $n$-dimensional random variable $z \in I_z$ following the distribution that is assumed known and that characterizes the random inputs. 
Here $Q$ is the collision operator describing binary collisions among particles with a bilinear form given by 
\begin{equation} \label{Qstrong}
Q(g,f)(v, z)=\int_{\bR^d}\int_{\Sd^{d-1}}B(|v-v_*|,\cos \theta, z)[g(v_*',z)f(v',z)-g(v_*,z)f(v,z)] \,\rd{\sigma}\, \rd{v_*}, 
\end{equation}
where $\sigma$ is a vector varying over the unit sphere $\Sd^{d-1}$, and $v'$, $v_*'$ are defined as
\begin{equation}\label{v'vs'}
v'=\frac{v+v_*}{2}+\frac{|v-v_*|}{2}\sigma, \quad v_*'=\frac{v+v_*}{2}-\frac{|v-v_*|}{2}\sigma.
\end{equation}
In this work, we only consider uncertainties coming from 
\begin{itemize}
\item[(i)] the initial datum $ f^{0}(v,z)$;
\item[(ii)] the collision kernel $B \geq 0$, which owns the form: 
\begin{equation}\label{kernel}
B(|v-v_*|,\cos\theta, z) = \Phi(|v-v_*|) b(\cos\theta, z), \qquad \cos\theta =\frac{\sigma\cdot (v-v_*)}{|v-v_*|},
\end{equation}
where the kinetic part $\Phi$ is a non-negative function, the angular part $b$ satisfies the Grad's cut-off assumption, that is, for all $z\in I_z$, 
\begin{equation}\label{cutoff}
\int_{\Sd^{d-1}} b(\cos\theta, z)\, \rd{\sigma}<\infty.
\end{equation}
One typical collision kernel that is widely used is the variable hard sphere (VHS) model \cite{HJ16}:
\begin{equation}\label{VHS}
    B(|v-v_*|,\cos\theta, z) = b(z) |v-v_*|^{\gamma},\quad 0 \leq \gamma \leq 1.
\end{equation}
\end{itemize}

We remark that the stochastic variable $z$ is a collection of random vectors with dimension $d_z$ following the distribution that is assumed known \textit{a priori}. For simplicity, we assume that all its components are mutually independent and already obtained from some dimension reduction technique, e.g., the Karhunen-Loeve expansion \cite{Xiu}. Moreover, the random vectors are not necessarily the same in different sources, i.e., the initial data and collision kernel, in our case. 

\subsection{Formulation of the numerical system}
\label{sec:review}

This paper primarily focuses on investigating the impact of random inputs on the deterministic approximation of the Boltzmann equation. However, this section not only delves into a comprehensive review of the deterministic numerical system for the spatially homogeneous Boltzmann equation, employing a Fourier-Galerkin spectral method in the velocity variable $v$, but also revisits the formulation of the numerical system in the stochastic variable. The latter involves a gPC-based stochastic Galerkin method in the random variable $z$. Additionally, we present a detailed exposition of a semi-discretized system that seamlessly involves both deterministic and stochastic approximations simultaneously, which is elucidated here to pave the way for our forthcoming work.

\subsubsection{A Fourier-Galerkin spectral method for velocity variable}
\label{subsec:spectralvelocity}

We provide a concise overview of the application of the Fourier-Galerkin spectral method to the deterministic Boltzmann equation, drawing primarily from \cite{PR00, HQ20}. For additional insights into the associated fast algorithms, one may refer to \cite{MP06, GHHH17, HM19}.

Consider an approximation of \eqref{IBE} on a periodic domain of velocity $\mathcal{D}_L=[-L,L]^d$, 
\begin{equation}\label{ABE}
\left\{
\begin{aligned}
&\partial_{t} f (t,v,z) = Q^{R}(f,f)(t,v,z), \quad t>0, \ v\in \mathcal{D}_L,\\[4pt]
& f(0,v,z)= f^{0}(v,z), 
\end{aligned}
\right.
\end{equation}
where the initial condition $f^0$ is a non-negative periodic function, $Q^{R}$ is the truncated collision operator of \eqref{Qstrong} defined by
\begin{equation}\label{QR}
\begin{split}
Q^R(g,f)(v,z)&=\int_{\mathcal{B}_R}\int_{\Sd^{d-1}}\Phi(|q|)\, b(\sigma\cdot \hat{q},\,z)\left[g(v'_*,z)f(v',z)-g(v-q,z)f(v,z)\right] \rd{\sigma}\, \rd{q}\\[4pt]
&=\int_{\bR^d} \int_{\Sd^{d-1}}\mathbf{1}_{|q|\leq R}\,\Phi(|q|)\, b(\sigma\cdot \hat{q},\,z)\left[g(v'_*,z)f(v',z)-g(v-q,z)f(v,z)\right]\rd{\sigma}\, \rd{q},
\end{split}
\end{equation}
where, after the change of variable $v_* \mapsto q=v-v_*$, the relative velocity $q$ is truncated to a ball $\mathcal{B}_R$ with radius $R$ centered at the origin. Denoting $q=|q|\hat{q}$ with $|q|$ the magnitude and $\hat{q}$ the directional unit vector and using \eqref{v'vs'}, we then re-write the post-collisional velocities $v',\, v_*'$ accordingly, 
\begin{equation}
v'=v-\frac{q-|q|\sigma}{2}, \quad v_*'=v-\frac{q+|q|\sigma}{2}.
\end{equation}
The values of $L$ and $R$ are positive satisfying $L \geq R > 0$,
where, in practice, to avoid aliasing errors, one usually takes
\begin{equation}\label{RL}
R=2S, \quad L\geq \frac{3+\sqrt{2}}{2}S,
\end{equation}
where the support set of $f^0(v,z)$ in $v$ is within $\mathcal{B}_S$ for all $z \in I_z$. We refer to \cite{PR00} for more details justifying this choice of $R$ and $L$ for anti-aliasing purposes.

We now seek a truncated Fourier series expansion of $f$ given as
\begin{equation}
f(t,v,z) \approx f_N(t,v,z) = \sum\limits_{|\n| = 0}^{N} f_{\n}(t,z) \Phi_{\n}(v)  \in \mathbb{P}_N \quad \text{with} \quad   \Phi_{\n}(v): = \e^{\im \frac{\pi}{L}\n\cdot v}, 
\end{equation}
where $N$ is a non-negative integer, $\n = (n_1,\dots,n_d)$ is a multi-index with $|\n| = n_1+\dots+n_d$, and the space
\begin{equation*}
\mathbb{P}_N=\text{span} \left\{ \Phi_{\n}(v) = \e^{\im \frac{\pi}{L} \n \cdot v} \, \Big| \, 0 \leq |\n| \leq N \right\} 
\end{equation*}
is equipped with the inner product in the velocity space: 
\begin{equation*}
\langle f(t,\cdot,z), g(t,\cdot,z) \rangle_{\D} = \frac{1}{(2L)^{d}}\int_{\D}f(t,v,z) \, \bar{g}(t,v,z) \,\rd v.
\end{equation*}
Substituting $f_N$ into \eqref{ABE} and conducting the Galerkin projection onto $\mathbb{P}_N$, one gets
\begin{equation} \label{PFS}
\left\{
\begin{aligned}
&\partial_{t} f_N(t,v,z) = \mP_N Q^{R}(f_N,f_N)(t,v,z), \quad t>0, \ v\in \mathcal{D}_L,\\[4pt]
& f_N(t=0,v,z)=f_{N}^{0}(v,z). 
\end{aligned}
\right.
\end{equation}
Here $\mP_N$ is the projection operator defined by: 
\begin{equation}\label{proj}
\mP_N g(t,v,z) = \sum_{|\n|=0}^{N} g_{\n}(t,z) \Phi_{\n}(v), \quad g_{\n}(t,z) = \langle g(t,\cdot,z), \Phi_{\n}(\cdot) \rangle_{\D},
\end{equation}
for any suitable function $g$. As a reasonable approximation of $f^0$, $f_N^0\in \mathbb{P}_N$ is set to be the initial data to the numerical system. More discussion on the initial condition will be given in subsection \ref{subsec:initial}. 

We write out the equation satisfied by each Fourier mode of \eqref{PFS}, 
\begin{equation}\label{FS}
\left\{
\begin{aligned}
&\partial_{t} f_{\n}(t,z) = Q^{R}_{\n}(f_{N},f_{N})(t,z), \quad  0 \leq |\n| \leq  N, \\[4pt]
& f_{\n}(t=0,z)= f^{0}_{\n}(z),
\end{aligned}
\right.
\end{equation}
with
\begin{equation*}
Q_{\n}^R:=\langle Q^R(f_N,f_N)(t,\cdot,z), \Phi_{\n}(\cdot) \rangle_{\D}, \quad f^0_{\n}(z):=\langle f_N^0(\cdot,z), \Phi_{\n}(\cdot)\rangle_{\D}.
\end{equation*}
By the definition \eqref{QR} and orthogonality of the Fourier basis, one derives that
\begin{equation} \label{sum}
Q_{\n}^R  = \sum\limits_{\substack{|\lb|,|\m|=0 \\ \lb+\m=\n}}^{N} G(\lb,\m)f_{\lb}f_{\m},
\end{equation}
where the weight $G$ is given by
\begin{equation}\label{GG}
\begin{split}
G(\lb,\m) & = \int_{\mathcal{B}_{R}}\int_{\Sd^{d-1}}\Phi(|q|)b(\sigma\cdot \hat{q}) \left[ \e^{-\im \frac{\pi}{2L}(\lb+\m)\cdot q +\im \frac{\pi}{2L}|q|(\lb-\m)\cdot \sigma} - \e^{-\im \frac{\pi}{L}\m\cdot q} \right] \rd\sigma\,\rd q\\[4pt]
& =\int_{\mathcal{B}_{R}}\e^{-\im \frac{\pi}{L}\m \cdot q}\left[\int_{\Sd^{d-1}}\Phi(|q|)b(\sigma\cdot \hat{q})(\e^{\im \frac{\pi}{2L}(\lb+\m)\cdot (q-|q|\sigma)}-1)\, \rd\sigma\right]\rd q. 
\end{split}
\end{equation}
The second equality presented above is derived by switching two variables $\sigma \leftrightarrow \hat{q}$ in the gain part of $G(\lb,\m)$. It is noteworthy that, within the framework of the direct Fourier spectral method, the computation of $G(\lb,\m)$ may only involve certain slow algorithms as well as require huge storage $O(N^{2d})$ to save. However, it is crucial to emphasize that this computation is a precomputed process, undertaken only once for a given collision kernel. Additionally, for a specific category of collision kernels, such as the VHS model \cite{GHHH17}, the weighted function $G(\lb,\m)$ can be simplified to a more concise form, facilitating easier evaluation.


\subsubsection{A gPC based stochastic Galerkin method for random variable}
\label{subsec:GalerkinStochastic}


We give a brief introduction to the implementation of the gPC based stochastic Galerkin method for the Boltzmann equation with uncertainties, mainly following \cite{HJ16, LJ18}.
In the generalized polynomial chaos approach in the stochastic Galerkin (referred by gPC-SG) framework \cite{Xiu,XK02}, one seeks the solution by the $K$-th order expansion as follows
\begin{equation}\label{fK}
	f(t,v,z) \approx \sum_{|\kk|= 0}^{K} f^{\kk}(t,v)\Psi^{\kk}(z) := f^{K}(t,v,z) \in \mathbb{P}^K,
\end{equation}
where $ \kk = (k_{1},...,k_{d_{z}}) $ is a multi-index with $ |\kk| = k_{1}+...+k_{d_{z}}$ and $\Psi^{\kk}(z)$ are the orthogonal gPC basis functions satisfying
\begin{equation*}
\int_{I_{z}} \Psi^{\ib}(z) \, \Psi^{\jb}(z)  \, \pi(z) \,\rd z = \delta_{\ib\jb}, \quad 0 \leq |\ib|, |\jb| \leq K. 
\end{equation*}
with $\pi(z)$ being the the probability distribution function of $z$.\\
Here the space 
 $\mathbb{P}^K := \text{Span} \left\{ \Psi^{\kk}(z) \, \Big|\, 0 \leq |\kk| \leq K \right\}$
is equipped with the inner product in the random space
\begin{equation*}
	\langle f(t,v,\cdot), \bar{g}(t,v,\cdot) \rangle_{I_{z}} = \int_{I_{z}}  f(t,v,z) g(t,v,z) \pi(z) \,\rd z.
\end{equation*}
Define $\mP^K$ as the projection operator in the random space: 
\begin{equation*}\label{projK}
	\mP^K g(t,v,z)=\sum_{|\kk|=0}^{K} g^{\kk}(t,v) \Psi^{\kk}(z), \quad g^{\kk}(t,v)=\int_{I_{z}} g(t,v,z) \Psi^{\kk}(z) \pi(z) \,\rd z, 
\end{equation*}
for any suitable function $g$. By inserting \eqref{fK} into \eqref{IBE} and performing the standard projection, one obtains the gPC-SG system
\begin{equation} \label{FSK}
	\left\{
	\begin{aligned}
		&\partial_{t} f^{\kk}(t,v) = Q^{\kk}(f^{K},f^{K})(t,v), \quad  0 \leq |\kk| \leq  K, \\[4pt]
		& f^{\kk}(t=0,v)= f^{0,\kk}(v),
	\end{aligned}
	\right.
\end{equation}
with
\begin{equation*}
	Q^{\kk}(f^{K},f^{K}):=\langle Q(f^K,f^K)(t,v,\cdot), \Psi^{\kk}(\cdot) \rangle_{I_{z}}, \quad f^{0,\kk}(v):=\langle f^K(0,v,\cdot), \Psi^{\kk}(\cdot) \rangle_{I_{z}}.
\end{equation*}

For a certain class of collision kernel owning the form \eqref{kernel}, $Q^{\kk}$ can be represented as
\begin{equation}
Q^{\kk}(t,v) = \sum_{|\ib|,|\jb| = 0}^{K} S_{\kk\ib\jb} \int_{\bR^d}\int_{\Sd^{d-1}}\Phi(|v-v_*|) \left[ f^{\ib}(v')f^{\jb}(v'_*)-f^{\ib}(v)f^{\jb}(v_*) \right] \,\rd\sigma \, \rd v_*
\end{equation}
with
\begin{equation}\label{S_kij}
S_{\kk\ib\jb} := \int_{I_{z}} b(z) \Psi^{\kk}(z) \Psi^{\ib}(z) \Psi^{\jb}(z) \pi(z) \,\rd z,
\end{equation}
where VHS collision kernel is taken as an example \cite{HJ16}, i.e., $b(\cos\theta, z) = b(z)$. Similar to the weight function $G(\lb,\m)$ as in \eqref{GG}, $S_{\kk\ib\jb}$ can be pre-computed and restored in advance as well, nevertheless, the evaluation of $Q^{\kk}(t,v)$ is by no means easy. In fact, if one takes a direct method without any fast algorithm, for each $t$ and $\kk$, it would lead to the huge computational cost in $O(N^2_{K}N^{d-1}_{\sigma}N^{2d}_v)$ with the binomial coefficient $N_K = \binom{K+n}{K}$, the number of discretized points in each angular direct $N_{\sigma}$, and the number of points in each velocity dimension $N_v$. 

For more details and related fast algorithm about applying the gPC based stochastic Galerkin method for the Boltzmann equation with uncertainties, we refer the readers to \cite{HJ16, LJ18} and the references therein.

\subsubsection{A discretized system in both velocity and random variables}
\label{subsec:coupled}

Now we are prepared to employ the Fourier-spectral expansion in the velocity variable and the gPC-based stochastic-Galerkin method \cite{Xiu, XK02} in the random space simultaneously, i.e., we seek an approximated solution in the following form: 
\begin{equation}\label{fNK}
	f(t,v,z) \approx \sum_{|\kk|= 0}^{K} \sum_{|\n|= 0}^{N} f^{\kk}_{\n}(t)\Phi_{\n}(v) \Psi^{\kk}(z) := f_{N}^{K}(t,v,z). 
\end{equation}
Inserting into \eqref{IBE} and conducting projections onto the space $\mathbb{P}_N$ and $\mathbb{P}^K$ successively yields
\begin{equation} \label{PFSNK}
	\left\{
	\begin{aligned}
		&\partial_{t} f_{N}^{K}(t,v,z) = \mP_N^K Q^{R}(f_N^K,f_N^K)(t,v,z), \quad t>0, \ v\in \mathcal{D}_L,\ z \in I_{z}, \\[4pt]
		& f_N^K(t=0,v,z)=f_{N}^{K}(0,v,z),
	\end{aligned}
	\right.
\end{equation}
where $\mP_N^K$ is the projection operator defined for any function $g$: 
\begin{equation}\label{projNK}
	\mP_N^K g(t,v,z) = \sum_{|\kk|= 0}^{K} \sum_{|\n|= 0}^{N}  g^{\kk}_{\n}(t) \Phi_{\n}(v) \Psi^{\kk}(z), \quad g_{\n}^{\kk}(t) = \Big \langle \langle g(t,\cdot,\cdot), \Phi_{\n}(\cdot) \rangle_{\D}, \Psi^{\kk}(\cdot) \Big\rangle_{I_{z}}. 
\end{equation}

Furthermore, we present this complete discretized (except for the discretization in the temporal space) system in its weak mode: 
\begin{equation} \label{FSNK}
	\left\{
	\begin{aligned}
		&\partial_{t} f^{\kk}_{\n}(t) = Q^{R,\kk}_{\n}(f^{K}_{N},f^{K}_{N})(t), \quad 0 \leq |\n| \leq  N, \,  0 \leq |\kk| \leq  K, \\[4pt]
		& f^{\kk}_{\n}(t=0,v,z)= f^{0,\kk}_{\n},
	\end{aligned}
	\right.
\end{equation}
with
\begin{equation*}
\begin{aligned}
 & Q^{R,\kk}_{\n}(f_{N}^{K},f_{N}^{K}):= \Big \langle \langle Q^{R}(f^{K}_{N},f^{K}_{N})(t,\cdot,\cdot), \Phi_{\n}(\cdot) \rangle_{\D}, \Psi^{\kk}(\cdot) \Big\rangle_{I_{z}}, \\[2pt]
& f^{0,\kk}_{\n}:= \Big \langle \langle f^K_{N}(0,\cdot,\cdot), \Phi_{\n}(\cdot) \rangle_{\D}, \Psi^{\kk}(\cdot) \Big\rangle_{I_{z}}.
\end{aligned}
\end{equation*}
In particular, we can further deduce the term $Q^{\kk}_{\n}$ as follows
\begin{equation}
	Q^{R,\kk}_{\n}(t) = \sum_{|\ib|,|\jb| = 0}^{K} S_{\kk\ib\jb} \sum\limits_{\substack{|\lb|,|\m|=0 \\[2pt] \lb+\m=\n}}^{N} G(\lb,\m)\, f^{\ib}_{\lb}\, f^{\jb}_{\m}
\end{equation}
with $S_{\kk\ib\jb}$ and $G(\lb,\m)$ given in \eqref{S_kij} and \eqref{GG} respectively. 


\section{Preliminary}
\label{sec:pre}

\subsection{Norms and Notations}
\label{subsec:norms}

We first introduce some norms and notations that will be used throughout the paper. 

For a function $f(t, v, z)$ that is periodic in velocity space $\D$, we define its Lebesgue norm and Sobolev norm with respect to the velocity variable:
\begin{equation*}
\|f(t,\cdot,z)\|^p_{L^p_{v}(\D)}:=\int_{\D} |f(t, v, z)|^p \,\rd{v}, \quad \|f(t, \cdot, z)\|^2_{{H^k_{v}}(\D)}:= \sum_{|\nu|\leq k} \|\partial_v^{\nu} f(t,\cdot,z) \|^2_{L^2_{v}(\D)}. 
\end{equation*}
For all $z \in I_z$, we define the Lebesgue and Sobolev norm for the $z$-derivatives of $f$ up to order $|l|\leq r$:
\begin{equation}
\begin{split}
& | f(t,\cdot,z) |_{L_v^p,r} := \sum_{|l|\leq r} \|\partial_z^l f(t,\cdot,z)\|_{L_v^p(\mathcal{D}_L)} \qquad  | f(t,\cdot,z) |_{H_v^k, r} := \sum_{|l|\leq r} \|\partial_z^l f(t,\cdot,z) \|_{H_v^k(\mathcal{D}_L)}. 
\end{split}
\end{equation}
In addition, one can take $\sup_{z \in I_z}$ above and define the following norms: 
\begin{equation}
	\begin{split}
		& \| f(t,\cdot,\cdot) \|_{L_v^p,r} := \sum_{|l|\leq r} \|\partial_z^l f(t,\cdot,\cdot)\|_{L_{v,z}^{p,\infty}} \qquad  \| f(t,\cdot,\cdot) \|_{H_v^k, r} := \sum_{|l|\leq r} \|\partial_z^l f(t,\cdot,\cdot) \|_{H_{v,z}^{k,\infty}}. 
	\end{split}
\end{equation}

Similar to the deterministic case, for a function $f(v, z)$ and each $z \in I_z$, 
the positive and negative parts are defined by 
\begin{equation}
f^+(v, z)=\max\limits_{v\in\D}\{ f(v, z), 0 \}, \quad 
f^-(v, z)=\max\limits_{v\in\D}\{ -f(v, z),0 \},
\end{equation}
so that $f=f^+-f^-$ and $|f|=f^++f^-$.

\subsection{Assumptions and known results}
\label{subsec:initial}

In this subsection, we introduce some basic assumptions that will be used in the following proof, especially pertaining to the collision kernel and initial conditions. Additionally, we elucidate known results, such as the mass conservation inherent in the Fourier-Galerkin spectral method.

\noindent{\bf Basic assumptions on the collision kernel:} 

\begin{itemize}

\item[(i)] Consider the collision kernel in the form \eqref{kernel}, with the kinetic part $\Phi$ satisfying
\begin{equation}\label{kinetic}
\left\| \mathbf{1}_{|q|\leq R}\Phi(|q|)\right\|_{L^{\infty}(\D)} < \infty . 
\end{equation}
Notice that the power law hard potentials $\Phi(|q|)=|q|^{\gamma}$ ($0\leq \gamma\leq 1$) and the "modified" soft potentials $\Phi(|q|)=(1+|q|)^{\gamma}$ ($-d<\gamma <0$) all satisfy this condition.


\item[(ii)] For all $z\in I_z$, we assume a uniform bound for all the $z$-derivatives of $b$: 
\begin{equation}\label{Assump_B}
 b(\sigma\cdot\hat{q},\,z) > 0, \qquad
 |\partial_z^k b(\sigma\cdot\hat{q},\,z)| \leq C_b \quad \text{for} \quad
 0\leq |k| \leq r. 
\end{equation}
\end{itemize}

\begin{remark}
We mainly focus on the collision kernel with uncertainties in the angular kernel $b(\cos\theta, z)$ in the work. In addition to the VHS model in \eqref{VHS}, one can approximate the collision kernel $b$ by the Karhunen-Loeve expansion \cite{Loeve1977}, given as
\begin{equation}
    b(\cos\theta,z) \approx b_0(\cos\theta) + \sum_{i=1}^{d_z} b_i(\cos\theta) z_i,
\end{equation}
with $z_1, \cdots, z_{d_z}$ independent random variables that follow the probability density function $\pi(z)$. Here we assume that all $b_i$, $i=1,\cdots,d_z$, are bounded to satisfy \eqref{Assump_B}.

While beyond the scope of this paper, it is noteworthy that our analysis can be similarly extended to the case where the kinetic part $\Phi$ in the collision kernel is assumed uncertain, i.e.,
$ B(|q|,\,\sigma\cdot\hat{q},\,z) = \Phi(|q|,\,z)\, b(\sigma\cdot\hat{q})$, if more stringent conditions can be carefully imposed on $\Phi(|q|,\,z)$ to ensure compatibility with \eqref{B_TZ}.

\end{remark}


\bigskip

\noindent{\bf Basic assumptions on the initial condition:} 

In order to prove the well-posedness and stability of the numerical solution to \eqref{PFS}, we need to restrict to a certain class of initial data. For the initial condition $f^0(v,z)$ to the original problem \eqref{ABE}, we assume it to be non-negative, periodic in the velocity space, and belongs to $L_v^1\cap H_v^1(\D)$ for all $z \in I_z$. For the approximated initial condition $f_N^0(v,z)=\mP_N f^0(v,z) $ to the numerical system \eqref{PFS}, we can show that for all $z \in I_z$ it satisfies the following properties \cite{HQY21}: 
\begin{itemize}
\item[(i)] Mass conservation of the approximation: for all $z \in I_z$,
\begin{equation} \label{con(a)}
\int_{\D} f^{0}_N(v,z)\, \rd v=\int_{\D} f^0(v,z)\, \rd v.
\end{equation}

\item[(ii)] Control of $L^2$ and $H^1$ norms: for all $z \in I_z$ and any integer $N\geq 0$,
\begin{equation} \label{con(b)}
\|f^0_N(\cdot,z)\|_{L_v^2(\D)}\leq \|f^0(\cdot,z)\|_{L_v^2(\D)}, \quad \|f^0_N(\cdot,z)\|_{H_v^1(\D)}\leq \|f^0(\cdot,z)\|_{H_v^1(\D)}.
\end{equation}

\item[(iii)] Control of $L^1$ norm: for all $z \in I_z$, there exists an integer $N_0$ such that for all $N > N_0$,
\begin{equation} \label{con(c)}
\|f_N^0(\cdot,z)\|_{L_v^1(\D)}\leq C \|f^0(\cdot,z)\|_{L_v^1(\D)}.
\end{equation}
where we take $C=2$ in the proof below.

\item[(iv)] $L^2$ norm of $f_N^{0,-}$ can be made arbitrarily small: for all $z \in I_z$ and any $\varepsilon>0$, there exists an integer $N_0$ such that for all $N > N_0$,
\begin{equation} \label{con(d)}
\|f_N^{0,-}(\cdot,z)\|_{L_v^2(\D)} < \varepsilon.
\end{equation}
\end{itemize}

\begin{remark}\label{rmk}
Note that $f_N^0=\mP_N f^0$ is not the only possible initial condition for the numerical solution, any reasonable numerical approximation satisfying the above assumption (i)--(iv) would work indeed. 
\end{remark}

It is known that one of the drawbacks of the Fourier-Galerkin spectral method is that some significant physical properties cannot be automatically preserved such as the positivity of numerical solution. However, one of the key properties -- the mass of the numerical solution $f_N$ -- is conserved for any time, which will provide useful control of $f_N$. To illustrate this, we recall the following preliminary Lemma,

\begin{lemma}\label{lemma:conv}
	The numerical system \eqref{PFS} preserves mass, i.e., 
	\begin{equation} 
	\int_{\D} f_N(t,v,z) \,\rd v = \int_{\D} f^{0}_N(v,z) \,\rd v, \quad \forall z \in I_z. 
	\end{equation}
\end{lemma}

\begin{remark}
    In the numerical system \eqref{PFS}, it is crucial to observe that only the velocity variable $v$ is discretized by the Fourier spectral method; consequently, for fixed $z\in I_z$, the proof is almost the same as the deterministic case in \cite[Lemma 2.1]{HQY21}. However, if both velocity $v$ and random variable $z$ are discriteized simultaneously, the mass conservation is expected to be much more complicated, which will be thoroughly investigated in our forthcoming work dealing with the numerical system \eqref{PFSNK}.
\end{remark}


\subsection{Estimates for the \texorpdfstring{$z$}{z}-derivatives of \texorpdfstring{$Q^{R}$}{QR}}
\label{subsec:QR}


The estimates of the collision operator $Q^R$ with uncertainty will play a key role in the proof for our main results, the gain term and loss term with the kernel satisfying \eqref{kernel} possess different structures and are defined as follows: for all $z \in I_z$,
\begin{equation}
\begin{split}
 Q^{R,+}(g,f)(v,z): &=\int_{\bR^d}\int_{\Sd^{d-1}}\mathbf{1}_{|q|\leq R}\,\Phi(|q|) \,b(\sigma\cdot \hat{q},z)\, g(v'_*,z)f(v',z)  \,\rd{\sigma}\rd{q}, \\
 Q^{R,-}(g,f)(v,z): &=\int_{\bR^d}\int_{\Sd^{d-1}}\mathbf{1}_{|q|\leq R}\,\Phi(|q|)\,b(\sigma\cdot \hat{q},z)\,g(v-q,z)f(v,z) \,\rd{\sigma} \rd{q}\\[4pt]
 & = f(v,z)L^R[g](v,z),
\label{Q_minus}
\end{split}
\end{equation}
where
$$ L^R[g](v,z):= \int_{\bR^d}\int_{\Sd^{d-1}}\mathbf{1}_{|q|\leq R}\,\Phi(|q|)\,b(\sigma\cdot \hat{q},z)\,g(v-q,z)\rd{\sigma}\rd{q}. $$

As a counterpart to the deterministic setting that merely depends on the velocity variable, one gets the estimates for the truncated collision operators $Q^{R,+}(g, f)(\cdot,z)$, $Q^{R,-}(g, f)(\cdot,z)$ and $Q^{R}(g, f)(\cdot,z)$ for all $z\in I_z$. We remark that with the assumption on the collision kernel \eqref{kernel}, the proof holds almost the same, except for the coefficient constants in the estimate bounds. For the reader's convenience, we put it in the Appendix. 

In the next step, we derive estimates for higher-order $z$-derivatives of the truncated collision operators $Q^{R,+}$, $Q^{R,-}$ and $Q^R$. 

\begin{proposition}
\label{QR_Z}
Let the truncation parameters $R$, $L$ satisfy \eqref{RL} and assume the collision kernel $B$ satisfy \eqref{kernel}--\eqref{cutoff} and  \eqref{kinetic}--\eqref{Assump_B}.
For all $z \in I_z$ and for any integer $|l|\leq r$, the $l$-th order $z$-derivative of the collision operators has the following bounds: 
\begin{equation}
\| \partial_z^l Q^{R,-}(g, f)(\cdot,z) \|_{L^p_v(\mathcal{D}_L)} \leq 
C_{R,L,d,r}^{-}(B)\, | g(\cdot,z) |_{L_v^1, r} \, | f(\cdot,z) |_{L_v^p, r}, 
\end{equation}
\begin{equation}
\| \partial_z^l Q^{R,+}(g, f)(\cdot,z) \|_{L^p_v(\mathcal{D}_L)} \leq 
C_{R,L,d,r}^{+}(B)\, | g(\cdot,z) |_{L_v^1, r} \, | f(\cdot,z) |_{L_v^p, r}, 
\end{equation}
\begin{equation}\label{Q_LpZ}
\| \partial_z^l Q^{R}(g, f)(\cdot,z) \|_{L^p_v(\mathcal{D}_L)} \leq 
C_{R,L,d,p,r}(B)\, | g(\cdot,z) |_{L^1_v, r} \, | f(\cdot,z) |_{L^p_v, r}. 
\end{equation}
Furthermore, 
\begin{equation}\label{Q_HkZ}
\| \partial_z^l Q^{R}(g, f)(\cdot,z) \|_{H^k_v(\mathcal{D}_L)} \leq 
C_{R,L,d,k,r}(B)\, | g(\cdot,z) |_{H^k_v, r} \, | f(\cdot,z) |_{H^k_v, r}. 
\end{equation}
\end{proposition}
\begin{proof}
Noticing the definition of $ Q^{R,-}(g,f)(v,z) $ in \eqref{Q_minus}, taking $\partial_z^l$ on both sides, and using the Leibniz rule, we have 
$$ \partial_z^l Q^{R,-}(g,f)(v,z) = \sum_{|n|\leq |l|} \binom{l}{n}\,
\partial_z^{l-n}\, f(v,z)\partial_z^n L^R[g](v,z). $$
In particular, we first look at the term
$$ \partial_z^n L^R[g](v,z)= \sum_{|m|\leq |n|}\binom{n}{m}\, \|\partial_z^m b(\cdot,z) \|_{L^1(\mathbb{S}^{d-1})}\left[\left(\mathbf{1}_{|q|\leq R}\Phi \right)\ast
\partial_z^{n-m}g\right](v,z). $$
Based on the assumption \eqref{Assump_B}, we know that $\| \partial_z^m b(\cdot,z) \|_{L^1(\mathbb{S}^{d-1})}$ is bounded, thus 
\begin{equation}
	\begin{split}
\| \partial_z^n L^R[g](\cdot,z) \|_{L_v^{\infty}(\mathcal{D}_L)} 
& \leq c_{r}(B)\, \sum_{|m|\leq |n|} \| \left[\left(\mathbf{1}_{|q|\leq R}\Phi\right) \ast
\partial_z^{n-m}g\right](\cdot,z) \|_{L_v^{\infty}(\mathcal{D}_L)} \\
& \leq c_{r}(B)\, \| \mathbf{1}_{|q|\leq R}\Phi(\cdot) \|_{L_v^{\infty}(\mathcal{D}_L)}
\sum_{|m|\leq |n|} \| \partial_z^{n-m}g(\cdot,z) \|_{L_v^1(\mathcal{B}_{\sqrt{2}L+R})} \\
& \leq  c_{R,L,d,r}(B) \sum_{|m|\leq |n|} \| \partial_z^{n-m} g(\cdot,z) \|_{L_v^1(\mathcal{D}_L)} \\
& \leq  c_{R,L,d,r}(B)\, |g(\cdot,z)|_{L_v^1,r}\,,  
\end{split}
\end{equation}
where $c_{r}(B)$, $c_{R,L,d,r}(B)$ are constants with subscripts indicating their parameters dependence. \\
Therefore, 
\begin{equation}
	\begin{split}
	\| \partial_z^l Q^{R,-}(g, f)(\cdot,z) \|_{L_v^p(\mathcal{D}_L)} 
	& = \left\| \sum_{|n|\leq |l|} \binom{l}{n}\, \partial_z^n L^R[g] (\cdot,z)
	\partial_z^{l-n} f(\cdot,z) \right\|_{L_v^p(\mathcal{D}_L)} \\
	& \leq C_r \sum_{|n|\leq |l|} \| \partial_z^n L^R[g](\cdot,z) \|_{L_v^{\infty}(\mathcal{D}_L)}\, | f(\cdot,z) |_{L_v^p, r} \\
	& \leq C_{R,L,d,r}^{-}(B)\, | g(\cdot,z) |_{L_v^1, r} \, | f(\cdot,z) |_{L_v^p, r}\,.
	\end{split}
\end{equation}
The proof for the estimates for $\partial_z^l Q^{R,+}$ and \eqref{Q_LpZ}--\eqref{Q_HkZ} can be done in a similar fashion, thanks to our assumption \eqref{Assump_B} for the collision kernel and the analysis in the deterministic setting (\cite[Appendix]{HQY21} and \cite[Theorem 2.1]{MV04}), where the derivation follows similarly in our case due to the boundedness of all orders of $z$-derivative of the collision kernel and the Leibniz rule. We thereby omit details. 
\end{proof}

As an extension, one can add up each order of $z$-derivative for the collision operator and obtain the following estimates in the norms that will be used in Section \ref{sec:main}. We summarize it in the Corollary below. 

\begin{corollary}
\label{sumQz}
Under the same assumptions in Proposition \ref{QR_Z}, we have the following estimates: 
\begin{equation}\label{QHKZ|}
\begin{split}
	| Q^{R}(g, f)(\cdot,z) |_{L^p_v,r} \leq 
	C^{'}_{R,L,d,p,r}(B)\, | g(\cdot,z) |_{L^1_v, r} \, | f(\cdot,z) |_{L^p_v, r}\,. 
\end{split}
\end{equation}
Furthermore, 
\begin{equation}\label{QHKZ||}
\begin{split}
	\| Q^{R}(g, f)(\cdot,\cdot) \|_{L^p_v,r} \leq 
	C_{R,L,d,p,r}^{''}(B)\, \| g(\cdot,\cdot) \|_{L^1_v, r} \, \| f(\cdot,\cdot) \|_{L^p_v, r}\,. 
\end{split}
\end{equation}
\end{corollary}
\begin{proof}
The proof is by adding up all $z$-derivatives of the collision operators for $|l|\leq r$, and we omit it here. 
\end{proof}


\section{Well-posedness of the numerical solution \texorpdfstring{$f_N$}{fN} with uncertainties}
\label{sec:main}

In this section, we establish the well-posedness of the numerical solution $f_N$, depending on the random variable $z$, in the Fourier-Galerkin spectral system \eqref{PFS} on any arbitrarily given time interval $[0, T]$. Note that, similar to the deterministic case in \cite{HQY21}, the main challenge in the proof lies in the fact that the numerical solution $f_N$ to the system~\eqref{PFS} is not necessarily non-negative because of the spectral projection in the velocity space,
even though its analytic counterpart $f$ to the original problem~\eqref{ABE} is always non-negative.

Our strategy is to extend the local well-posedness result to the whole time interval $[0, T]$, by applying the energy-type method. In Section~\ref{subsec:propagation}, we obtain the propagation of $L_v^2$ and $H_v^k$ norm of the solution $f_N$ with uncertainties, where the key is a "good" control of the negative part of $f_N$, besides using the mass conservation and {\it a priori} $L^1$ bound. In Section~\ref{subsec:wellposed}, first the local existence and uniqueness over a sufficiently small time interval $[0, \tau]$ will be shown via a fixed point theorem, where the negative part of the numerical solution $f_N$ is controlled over the same period of time with large enough $N$, which further implies that the initial $L_v^1$ bound of $f_N$ can be preserved at time $\tau$. Therefore, the same procedure can be repeated iteratively to extend the solution up to the prescribed final time $T$, where the values of designated parameters $N$ and $\tau$ are shown to be the same in each iteration as the beginning ones.

\subsection{Propagation of \texorpdfstring{$H_v^k$}{Hvk} estimates with uncertainties}
\label{subsec:propagation}

In this subsection, we first recall the $H_v^k$ estimates of $f_N$ studied in \cite{HQY21} with {\it a priori} $L^1$ bound of $f_N$, where no uncertainty is involved. Note that for each $z \in I_z$, under our assumptions on the collision kernel, the estimates for $f_N$ hold similarly. We thereby put it in the Appendix. 

\subsubsection{\texorpdfstring{$H_v^k$}{Hvk} estimate of \texorpdfstring{$f_N$}{fN} with \texorpdfstring{$z$}{z}-derivatives}
\label{subsub:Hk}


We now consider $z$-derivatives of $f_N$. The strategy in obtaining local well-posedness results for each order $z$-derivative of $\fe$ is reminiscent of \cite{DJL21}.
\begin{proposition}
\label{regularity}
Let the truncation parameters $R$, $L$ satisfy \eqref{RL} and the collision kernel $B$ satisfy \eqref{kernel}--\eqref{cutoff} and  \eqref{kinetic}--\eqref{Assump_B}.
For the numerical system \eqref{PFS}, assume that the initial condition $\partial_z^l f_{N}^0(v,z)  \in H^{k}_{v}(\D) $ for some integers $k\geq 0$ and $|l|\leq r$, that is,  
\begin{equation}
    \| \partial^{l}_{z}f^{0}(\cdot,z)\|_{H^{k}_{v}(\D)} \leq D^{f^0}_{k,l}, \qquad \forall z \in I_{z}.
\end{equation}
If each $z$-derivative of the solution $f_N(t,v,z)$ is assumed to have a $L_v^{1}$ bound up to some time $t_0$, i.e., 
\begin{equation}\label{Priori_fZ}
\forall t\in [0,t_0], \quad \left\| \partial_z^l \fe(t,\cdot,z) \right\|_{L_{v}^{1}(\D)} \leq  E_{l}, \qquad \forall z \in I_{z}, 
\end{equation}
then there exists a constant $K_{k,l}$ depending on $t_0$, $E_l$, $D^{f^0}_{k,l}$ and $C_{R,L,d,2}(B)$ such that, for $|l|\leq r$,
\begin{equation}\label{fN_HkZ}
\forall t\in[0,t_0], \quad \left\| \partial_z^l \fe(t,\cdot,z) \right\|_{H^k_v(\D)} \leq K_{k,l}(t_0), \qquad \forall z \in I_z \,.
\end{equation}
\end{proposition}

\begin{proof}
Recall the truncated collision operator $Q^R$ in \eqref{QR} with the uncertain collision kernel, 
$$ Q^R(h,f)(v,z) = \int_{\mathbb R^d}\int_{\mathbb{S}^{d-1}} B_T(\sigma\cdot \hat{q},\,|q|,\,z) \left[h(v_*',z)f(v',z)-h(v_{*},z)f(v,z)\right] \,\rd{\sigma}\, \rd{v_{*}}\,, $$
where $B_T(\sigma\cdot \hat{q},|q|,z)$ is denoted as the truncated collision kernel $B_T = \mathbf{1}_{|q|\leq R}\,\Phi(|q|)\,b(\sigma\cdot \hat{q}, z)$. 
With \eqref{Assump_B}, we assume that for $|l|\leq r$, all $l$-th order $z$-derivative of $B_T$ is uniformly bounded, i.e., 
\begin{equation}\label{B_TZ} 
| \partial_z^l B_T | \leq C, \qquad \forall z \in I_z. 
\end{equation}
For the convenience of notation, let us define 
\begin{equation}
\begin{split}
    Q_{B^{k}}^R(h,f)(v,z) =& \int_{\mathbb R^d}\int_{\mathbb{S}^{d-1}} \partial_z^k B_T(\sigma\cdot \hat{q},|q|,z) \left[h(v_*',z)f(v',z) - h(v_{*},z)f(v,z)\right]\,\rd{\sigma}\, \rd{v_{*}}\\[4pt]
    :=& \int_{\mathbb R^d}\int_{\mathbb{S}^{d-1}} \partial_z^k B_T \left[h'_{*} f' - h f \right]\,\rd{\sigma}\, \rd{v_{*}}\,.
\end{split}
\end{equation}
Note that the only difference between $Q_{B^{k}}^R$ and $Q^R$ is that $B_T$ is substituted by $\partial_z^l B_T$. Under the assumption \eqref{B_TZ}, it is obvious that the estimates for $Q^R$ in the deterministic problem \cite{HQY21} follow similarly for $Q_{B^{k}}^R$ here, in particular, 
\begin{equation}\label{Q_BZ} 
	\|Q_{B^{k}}^R(h,f)(\cdot,z)\|_{L_v^p(\D)} \leq C_{R,L,d,p}(B)\, \|h(\cdot,z)\|_{L_v^1(\D)}\, \|f(\cdot,z)\|_{L_v^p(\D)} \,. 
\end{equation}
By using the Leibniz rule, one can calculate that
\begin{equation}\label{Q_Leibniz}
\begin{split}
\partial_z^l Q^R(h, f) & = \sum_{|n|=0}^{|l|} \binom{l}{n} \int_{\mathbb R^d}\int_{\mathbb{S}^{d-1}}
\partial_z^{l-n}B_T \sum_{|m|=0}^{|l|} \binom{n}{m}\left[\partial_z^m h_{*}'\, \partial_z^{n-m}f' - \partial_z^m h_{*}\,\partial_z^{n-m}f\right]\,\rd{\sigma}\, \rd{v_{*}} \\[4pt]
& = \sum_{|n|=0}^{|m|} \binom{n}{l} \sum_{|m|=0}^{|l|} \binom{n}{m} Q_{B^{n-l}}^R (\partial_z^m h, \partial_z^{n-m}f) \\[4pt]
& = \sum_{|n|=0}^{|l|-1} \sum_{|m|=0}^{|n|} \binom{l}{n} \binom{n}{m} Q_{B^{l-l}}^R (\partial_z^m h, \partial_z^{n-m}f)
+ \sum_{|m|=1}^{|l|-1}\binom{l}{m}Q^R(\partial_z^m h, \partial_z^{l-m}f) \\[4pt] 
& \quad + Q^R(h, \partial_z^l f) + Q^R(\partial_z^l h, f)\,.  
\end{split}
\end{equation}
Let $h$ and $f$ both equal to $f_N$, we prepare for the bound that will be used later:  
\begin{equation}\label{Q_Leibniz2}
\begin{split}
&\| \partial_z^l Q^R(f_N, f_N)(\cdot,z) \|_{L_v^2(\D)} \\[4pt]
 \leq&
C_{R,L,d,2}(B) \sum_{|n|=0}^{|l|-1} \sum_{|m|=0}^{|n|} \binom{l}{n} \binom{n}{m} \|\partial_z^m f_N(\cdot,z) \|_{L_v^1(\D)} \,
\|\partial_z^{n-m}f_N(\cdot,z) \|_{L_v^2(\D)} \\[4pt]
& \quad + C_{R,L,d,2}(B) \sum_{|m|=1}^{|l|-1} \binom{l}{m}  \|\partial_z^m f_N(\cdot,z) \|_{L_v^1(\D)}\, \|\partial_z^{l-m} f_N(\cdot,z) \|_{L_v^2(\D)} \\[4pt]
&\quad + C_{R,L,d,2}(B)\, \|f_N(\cdot,z)\|_{L_v^1(\D)}\, \|\partial_z^l f_N(\cdot,z) \|_{L_v^2(\D)} \\[4pt]
& \quad + C_{R,L,d,2}(B)\, \|\partial_z^l f_N(\cdot,z)\|_{L_v^1(\D)}\, \|f_N(\cdot,z)\|_{L_v^2(\D)}\,,
\end{split}
\end{equation}
where \eqref{Q_BZ} is used for each term on the right-hand-side with $p=2$. 

We now prove the Proposition by induction. When $|l|=0$, for all $z \in I_z$ the proof is exactly the same as the deterministic case, thus for $ t\in [0,t_0]$ and for all $z \in I_z$, $\left\| \fe(t,\cdot,z) \right\|_{L^{2}_{v}(\D)} \leq K_{0,0}$. 
Assume that $\left\| \partial_z^m \fe(t,\cdot,z) \right\|_{L^{2}_{v}} \leq K_{0,m}$ for all integers $|m| \leq |l|-1$ and $z \in I_z$. 
Take $\partial_z^l$ on both sides of Eq.~\eqref{PFSNK}, then 
\begin{equation}
	\partial_t \partial_z^l f_N(t,v,z) = \mathcal{P}_N \partial_z^l Q(f_N, f_N)(t,v,z).
\end{equation}
Multiplying both sides by $\partial_z^l f_N$ and integrating over $\mathcal{D}_L$ gives 
\small{
\begin{equation}\label{fN_inequality}
\begin{split}
& \quad \frac{1}{2} \frac{\rd}{\rd t} \|\partial_z^l f^N(t,\cdot,z) \|_{L_{v}^2(\D)}^2  = \int_{\D} \mathcal{P}_N \partial_z^l Q^R(f_N, f_N) (t,v,z)
\partial_z^l f_N(t,v,z) \,\rd v \\
& \leq \|\mathcal{P}_N \partial_z^l Q^R(f_N, f_N)(t,\cdot,z) \|_{L_v^2(\D)}\,
\|\partial_z^l f_N(t,\cdot,z) \|_{L_v^2(\D)} \\[4pt]
& \leq \| \partial_z^l Q^R(f_N, f_N)(t,\cdot,z) \|_{L_v^2(\D)}\, \|\partial_z^l f_N(t,\cdot,z) \|_{L_v^2(\D)} \\[4pt]
&  \leq C_{R,L,d,2}(B) \Big( E_0 \, \|\partial_z^l f_N(t,\cdot,z) \|_{L_v^2(\D)} + E_l \, \|f_N(t,\cdot,z)\|_{L_v^2(\D)} + C_{E,K,l-1} \Big)\, \|\partial_z^l f_N(t,\cdot,z) \|_{L_v^2(\D)} \\[4pt]
&  \leq C_{R,L,d,2}(B)\, E_0\, \|\partial_z^l f_N(t,\cdot,z) \|_{L_v^2(\D)}^2 + C_{R,L,d,2}(B) \left( E_l K_{0,0} + C_{E,K,l-1}\right) \|\partial_z^l f_N(t,\cdot,z) \|_{L_v^2(\D)}\,,
\end{split}
\end{equation}
}
where {\it a priori} $L_v^1$ bound of $\partial_z^l f_N$, given in Eq.~\eqref{Priori_fZ}, and the induction assumption is used in the third inequality. Here $C_{E,K,l-1}$ represents the constant depending on $l-1$, maximum of $E_m$ and $K_{0,m}$ for $|m| \leq |l|-1$. For notation simplicity, we will shorten the subscript of constants. Use the Gr{\"o}nwall's inequality, for all $z\in I_z$, then 
\begin{equation}
\begin{split} 
\| \partial_z^l f_N(t,\cdot,z) \|_{L_{v}^2(\D)} &\leq  \e^{2 C_{R,L,d,2,E_0}(B)\, t} \left( \|\partial_z^l f_N^0(t,\cdot,z) \|_{L_{v}^2(\D)} + C_{R,L,d,2,E,K,l-1}(B) \right)   \\[4pt]
&\leq \e^{2 C_{R,L,d,2,E_0}(B)\, t} \left( D_{0,l}^{f^0} + C_{R,L,d,2,E,K,l-1}(B)\right) : = K_{0,l}(t). 
\end{split}
\end{equation}
By mathematical induction, we can prove that for all $|l|\leq r$ and $z \in I_z$, 
$$\forall t \in [0, t_0],\quad  \left\| \partial_z^l f_N(t, \cdot, z) \right\|_{L^2_v(\D)} \leq K_{0,l}(t_0)\,. $$

To extend the estimate to $H_v^k(\D)$ norm of $\partial_z^l f_N$, 
we can obtain the counterpart of \eqref{fN_inequality}, i.e., for all $z\in I_z$ and $\forall t \in [0, t_0]$,
$$\frac{1}{2} \frac{\rd}{\rd t}\left\|\partial_z^l \fe(t,\cdot,\cdot)\right\|^{2}_{H_v^{k+1}(\D)} \leq \left\|\partial_z^l Q^{R}(\fe,\fe)(t,\cdot,\cdot) \right\|_{H_v^{k+1}(\D)} \left\| \partial_z^l \fe(t,\cdot,\cdot) \right\|_{H_v^{k+1}(\D)}, $$
where $|l|\leq r$. By combining the the estimate \eqref{Q_HkZ} in Proposition~\ref{QR_Z}, Eq.~\eqref{Q_Leibniz} and induction, the rest of the proof follows naturally, thus we omit it here. 
\end{proof}
After studying each order of $z$-derivatives for the numerical solution $\fe$, we extend the regularity result to our newly-established norm wherein all the $z$-derivatives of $\fe$ are added up. 
\begin{corollary}\label{regularity1}
	Let the truncation parameters $R$, $L$ satisfy \eqref{RL} and the collision kernel $B$ satisfy \eqref{kernel}--\eqref{cutoff} and \eqref{kinetic}--\eqref{Assump_B}.
	For the numerical system \eqref{PFS}, if the initial datum is assumed to be
	\begin{equation}
		| f^{0}(\cdot,z)|_{H^{k}_{v},r} \leq D^{f^0}_{k,r},  \quad \forall z \in I_{z}
	\end{equation}
	which is equivalent to $\| f^{0}(\cdot,\cdot)\|_{H^{k}_{v},r} \leq D^{f^0}_{k,r}$. If we further assume
	\begin{equation}
		\forall t \in [0,t_0], \quad |f_N(t,\cdot,z)|_{L^{1}_{v},r} \leq E_{r}, \quad \forall z \in I_{z} 
	\end{equation}
	which also implies $\|f_N(t,\cdot,\cdot)\|_{L^{1}_{v},r} \leq E_{r}$.
	Then, for any $ t\in[0,t_0]$, we have, 
	\begin{equation}\label{fN_HkZr}
	\begin{split}
	    |f_N(t,\cdot,z) |_{H^{k}_{v},r} \leq K_{k,r}(t_0), \quad \text{or equivalently} \quad & \| f_N(t,\cdot,\cdot) \|_{H^{k}_{v},r} \leq K_{k,r}(t_0), \quad \forall z \in I_{z}
	\end{split}
	\end{equation}
 where $K_{k,r}$ is a constant depending on $t_0$, $E_r$, $D^{f^0}_{k,r}$ and $C_{R,L,d,2,r}(B)$. 
\end{corollary}


\subsubsection{\texorpdfstring{$L_v^2$}{Lv2} estimate of the negative part \texorpdfstring{$f^-_N$}{f-N} with \texorpdfstring{$z$}{z}-derivatives}


We now proceed to estimate the negative part of $z$-derivatives of $\fe$, i.e., $\partial_z^l \fe^{-}$, which relies on a careful estimate of both the gain and loss terms of the collision operator. This estimate will play a key role in the main Theorem \ref{existencetheorem} for the well-posedness of $f_N$.

\begin{proposition}
Let the truncation parameters $R$, $L$ satisfy \eqref{RL}, and assume the collision kernel $B$ satisfy \eqref{kernel}--\eqref{cutoff} and \eqref{kinetic}--\eqref{Assump_B}. 
For the negative part of the numerical solution $f_N^-$ to the system \eqref{PFS}, one has
\begin{equation}\label{fe_minus}
\forall t \in [0,t_0],\quad  \left\| f_N^-(t,\cdot, z) \right\|_{L^2_{v},r}
 \leq \e^{ \mathcal{C}_{K_{0,r}}\, t} \left( \| f_N^{0,-}(\cdot, z) \|_{L^2_{v},r} + \frac{\mathcal{C}_{K_{1,r}}}{N} \right), 
\end{equation}
where $\fe^{0,-}$ stands for the initial datum of $\fe^{-}$, the constant $\mathcal{C}_{K_{0,r}}$ depends on $R$,$L$,$d$,$B$,$E_r$,$K_{0,r}$, 
while the constant $\mathcal{C}_{K_{1,r}}$ depends on 
$R$,$L$,$d$,$r$,$B$ and $K_{1,r}$.
\end{proposition}

\begin{proof}

Denote $\gn^{l}(t,v,z) = \partial_z^l \fe(t,v,z)$ for $|l|\leq r$. Take $l$-th order $z$-derivative on both sides of the equation 
\begin{equation}
    \partial_t \fe(t,v,z) = \mathcal{P}_N Q^R(\fe, \fe)(t,v,z),
\end{equation}
then $g_N^l$ satisfies the numerical system
\begin{equation}
	\partial_t g_{N}^{l}(t,v,z) = \mP_N \partial_z^l Q^R(\fe,\fe)(t,v,z).
\end{equation}

Note that $\gn^{l}(t,v,z) = \gn^{l,+}(t,v,z) - \gn^{l,-}(t,v,z)$. We first rewrite the above equation as 

\begin{equation} \label{eqn-g1} 
\partial_t \gn^{l}(t,v,z) = \partial_z^l Q^R(\fe, \fe)(t,v,z) + E^{l}_N(t,v,z), 
\end{equation}
where 
\begin{equation} \label{eqn-E}
E_N^{l}(t,v,z):= \mP_N \partial_z^l Q^R(\fe, \fe)(t,v,z)  - \partial_z^l Q^R(\fe, \fe)(t,v,z). 
\end{equation}

Define the indicator function $\mathbf{1}_{\left\lbrace \gn^{l}\leq 0\right\rbrace}$ as follows: 
\begin{equation}\label{chi1}
	\mathbf{1}_{\left\lbrace \gn^{l}\leq 0\right\rbrace} (v) := \begin{cases}
		1,& \text{for}\ v\ \text{s.t.}\ \gn^{l}(\cdot,v,\cdot) \leq 0, \\
		0,& \text{for}\ v\ \text{s.t.}\ \gn^{l}(\cdot,v,\cdot) > 0.
	\end{cases}
\end{equation}

Multiply both sides of \eqref{eqn-g1} by $\gn^{l}\mathbf{1}_{\left\lbrace \gn^{l}\leq 0\right\rbrace}$ and integrate on $\D$, the left-hand-side becomes 
\begin{equation}\label{gn_char} 
\begin{split}
    \gn^{l}  \mathbf{1}_{\left\lbrace \gn^{l}\leq 0\right\rbrace} (t,v,z) \partial_t \gn^{l}(t,v,z) =& - \gn^{l,-}(t,v,z) \partial_t \left[\gn^{l,+}(t,v,z) - \gn^{l,-}(t,v,z)\right]\\
    =& \ \gn^{l,-}(t,v,z) \partial_t \gn^{l,-}(t,v,z),
\end{split}
\end{equation}
whereas, for the right-hand-side, we first estimate the remainder $E^{l}_N$, 
\begin{equation}
\begin{split}
\| E^{l}_N(t,\cdot,z) \|_{L_v^2(\D)} & = \left\| \mP_N \partial_z^l Q^R(\fe, \fe)(t,\cdot,z)  - \partial_z^l Q^R(\fe, \fe)(t,\cdot,z) \right\|_{L_{v}^2(\D)} \\[4pt]
& \leq \frac{C_{\mP}}{N} \left\| \partial_z^l Q^R(\fe, \fe)(t,\cdot,z) \right\|_{H_{v}^1(\D)}\\
& \leq \frac{C_{\mP}C_{R,L,d,1,r}(B)}{N} \Big | \fe(t,\cdot,z) \Big |^2_{H_{v}^1, r}\,,
\end{split}
\end{equation}
where we used the property of the projection operator $\mP_N$ and the estimate \eqref{Q_HkZ} in the last inequality. 

\vspace{0.1in}
By considering \eqref{fN_HkZr}, since $\|\partial_z^l \fe(t,\cdot,z) \|_{H_{v}^1(\D)}$ is bounded, so we have 
$$\displaystyle\big | f_N(t,\cdot,z) \big |_{H_{v}^1, r} = \sum_{|l|\leq r} \| \partial_z^l f_N(t,\cdot,z) \|_{H_{v}^1(\D)} \leq K_{1,r}(t),$$ 
Therefore, 
\begin{equation}\label{EN}
\begin{split}
\int_{\D} E^{l}_N(t,v,z) \gn^{l} \mathbf{1}_{\left\lbrace \gn^{l}\leq 0\right\rbrace}(t,v,z) \, \rd v & = - \int_{\D} E^{l}_N(t,v,z) \gn^{l,-}(t,v,z) \,\rd v \\[4pt]
& \leq \| E^{l}_N(t,\cdot,z) \|_{L_{v}^2(\D)}\, \| \gn^{l,-}(t,\cdot,z) \|_{L_{v}^2(\D)} \\[4pt]
& \leq \frac{C_{\mP}C_{R,L,d,1,r}(B)}{N} \Big | \fe(t,\cdot,z) \Big |^2_{H_{v}^1, r}\, \|\gn^{l,-}(t,\cdot,z)\|_{L_{v}^2(\D)} \\[4pt]
& \leq  \frac{C_{\mP}C_{R,L,d,1,r}(B)\, K^2_{1,r}(t)}{N} \|\gn^{l,-}(t,\cdot,z)\|_{L_{v}^2(\D)}\,.
\end{split}
\end{equation}  
We now estimate the other term on the right-hand side, 
\begin{equation}\label{Q_g}
\begin{aligned}
&\int_{\D} \partial_z^l Q^R(\fe, \fe)(t,v,z) \gn^{l}\mathbf{1}_{\left\lbrace \gn^{l} \leq 0\right\rbrace}(t,v,z) \, \rd v \\[4pt]
= & - \int_{\D} \partial_z^l Q^R(\fe, \fe)(t,v,z) \gn^{l,-}(t,v,z) \,\rd v \\[4pt]
\leq & \|\partial_z^l Q^R(\fe, \fe)(t,\cdot,z)\|_{L_{v}^2(\D)} \| \gn^{l,-}(t,\cdot,z) \|_{L_{v}^2(\D)}. 
\end{aligned}
\end{equation}  

By the calculation in \eqref{Q_Leibniz}, note that
\begin{equation}\label{QR_Z_split}
 \partial_z^l Q^R(\fe, \fe)(t,v,z) = \text{ L.O.T } + Q^R(\gn^{l}, \fe)(t,v,z) + Q^R(\fe, \gn^{l})(t,v,z), 
\end{equation}  
where \text{``L.O.T"} stands for lower order terms that involve lower order of $z$-derivatives less than $l$.  
Based on \eqref{Q_Leibniz}--\eqref{Q_Leibniz2} and Proposition \ref{regularity} telling us all orders of $\partial_z^l \fe$ is bounded, shown by inequality \eqref{fN_HkZ}, it is obvious to see that $\| \text{ L.O.T } \|_{L^2(\D)} \leq C_{R,L,d,2}(B)\, C_{E,K,l-1}$, with $C_{E,K,l-1}$ 
the same as in \eqref{fN_inequality}. 

Let us now look at the second term on the right-hand-side of Eq.~\eqref{QR_Z_split}, namely
$$Q^R(\gn^{l}, \fe) = Q^{R,+}(\gn^{l}, \fe) - Q^{R,-}(\gn^{l},\fe).$$ For the gain term, multiplying $\gn^{l} \mathbf{1}_{\left\lbrace \gn^{l} \leq 0\right\rbrace}$ and integrating on $\D$ gives us: 
\begin{equation}
\label{QR_plus}
\begin{split}
&\int_{\D} Q^{R,+}(\gn^{l}, \fe)(t,v,z) \gn^{l} \mathbf{1}_{\left\lbrace \gn^{l} \leq 0\right\rbrace}(t,v,z) \,\rd v\\
=& \int_{\D} Q^{R,+}(\gn^{l,+} - \gn^{l,-}, \fe^{+} - \fe^{-})(t,v,z) \gn\mathbf{1}_{\left\lbrace \gn\leq 0\right\rbrace}(t,v,z) \, \rd v \\[2pt]
=& \int_{\D} \left[ Q^{R,+}(\gn^{l,+}, \fe^{+}-\fe^{-}) - Q^{R,+}(\gn^{l,-}, \fe^{+}-\fe^{-})\right](t,v,z) \gn^{l}\mathbf{1}_{\left\lbrace \gn^{l}\leq 0\right\rbrace}(t,v,z) \, \rd v \\[2pt]
=& \int_{\D} \Big[ Q^{R,+}(\gn^{l,+}, \fe^{+}) - Q^{R,+}(\gn^{l,+}, \fe^{-}) - Q^{R,+}(\gn^{l,-},\fe^{+})\\
&\qquad \qquad\qquad \qquad\qquad \qquad\qquad\qquad \qquad\qquad+ Q^{R,+}(\gn^{l,-}, \fe^{-})\Big](t,v,z)\, ( -\gn^{l,-} )(t,v,z) \, \rd v \\[2pt]
=& \int_{\D}\Big[ - Q^{R,+}(\gn^{l,+}, \fe^{+}) + Q^{R,+}(\gn^{l,+}, \fe^{-}) + Q^{R,+}(\gn^{l,-},\fe^{+}) \\[4pt]
&\qquad \qquad\qquad \qquad\qquad \qquad\qquad\qquad \qquad\qquad- Q^{R,+}(\gn^{l,-}, \fe^{-})\Big](t,v,z)\, \gn^{l,-}(t,v,z) \, \rd v 
\end{split}
\end{equation}
so the gain part can be further estimated as follows:
\begin{equation}
    \begin{split}
    &\int_{\D} Q^{R,+}(\gn^{l}, \fe)(t,v,z) \gn^{l} \mathbf{1}_{\left\lbrace \gn^{l} \leq 0\right\rbrace}(t,v,z) \,\rd v\\[4pt]
        \leq & \int_{\D}\left[ Q^{R,+}(\gn^{l,+}, \fe^{-}) + Q^{R,+}(\gn^{l,-},\fe^{+}) \right](t,v,z)\, \gn^{l,-}(t,v,z)\, \rd v \\[4pt]
    \leq &\ \| Q^{R,+}(\gn^{l,+}, \fe^{-})(t,\cdot,z) \|_{L_{v}^2(\D)} \| \gn^{l,-}(t,\cdot,z) \|_{L_{v}^2(\D)} \\[4pt]
    &\qquad\qquad + \| Q^{R,+}(\gn^{l,-},\fe^{+})(t,\cdot,z) \|_{L_{v}^2(\D)} \| \gn^{l,-}(t,\cdot,z) \|_{L_{v}^2(\D)} \\[4pt]
    \leq &\  C_{R,L,d,2}^{+}(B)\, \| \gn^{l,+}(t,\cdot,z) \|_{L_{v}^1(\D)} \| \fe^{-}(t,\cdot,z) \|_{L_{v}^2(\D)} \| \gn^{l,-}(t,\cdot,z) \|_{L_{v}^2(\D)} \\[4pt]
    & \qquad\qquad+ C_{R,L,d,2}^{+}(B)\, \| \gn^{l,-}(t,\cdot,z) \|_{L_{v}^1(\D)} \|\fe^{+}(t,\cdot,z)\|_{L_{v}^2(\D)} \| \gn^{l,-}(t,\cdot,z) \|_{L_{v}^2(\D)} \\[4pt]
    \leq &\ C_{R,L,d,2}^{+}(B)\, \| \gn^{l}(t,\cdot,z) \|_{L_{v}^1(\D)}
    \left| \fe^{-}(t,\cdot,z) \right|^{2}_{L_{v}^2,r} \\[4pt]
    & \qquad\qquad+ \tilde C_{R,L,d,2}^{+}(B)\, \| \fe(t,\cdot,z) \|_{L_{v}^2(\D)} \| \gn^{l,-}(t,\cdot,z) \|^{2}_{L_{v}^2(\D)}, 
    \end{split}
\end{equation}
where \eqref{QGLp1} for the gain operator is used in the second before last inequality. 
For the loss term, we have the decomposition
\begin{equation} \label{QR_neg}
\begin{split}
&- \int_{\D} Q^{R,-}(\gn^{l}, \fe)(t,v,z)  \gn^{l} \mathbf{1}_{\left\lbrace \gn\leq 0\right\rbrace}(t,v,z)\, \rd v \\[4pt]
=& -\int_{\D} L^{R}[\gn^{l}](t,v,z)\fe(t,v,z)\gn^{l,-}(t,v,z)\, \rd v \\[4pt]
=& -\int_{\D} \left( L^{R}[\gn^{l,+}](t,v,z)\fe(t,v,z)\gn^{l,-}(t,v,z) -L^{R}[\gn^{l,-}](t,v,z)\fe(t,v,z)\gn^{l,-}(t,v,z)  \right)\,\rd v \\[4pt]
=& -\int_{\D} \Big( L^{R}[\gn^{l,+}](t,v,z)\fe^{+}(t,v,z)\gn^{l,-}(t,v,z) - L^{R}[\gn^{l,+}](t,v,z)\fe^{-}(t,v,z)\gn^{l,-}(t,v,z) \Big) \,\rd v \\[4pt]
& + \int_{\D}L^{R}[\gn^{l,-}](t,v,z)\fe(t,v,z)\gn^{l,-}(t,v,z)  \,\rd v
\end{split}
\end{equation}
then, it can be further estimated 
\small{
\begin{equation}
    \begin{split}
    &- \int_{\D} Q^{R,-}(\gn^{l}, \fe)(t,v,z)  \gn^{l} \mathbf{1}_{\left\lbrace \gn\leq 0\right\rbrace}(t,v,z)\, \rd v \\[4pt]
        \leq & \int_{\D} L^{R}[\gn^{l,+}](t,v,z)\fe^{-}(t,v,z)\gn^{l,-}(t,v,z)\,\rd v + \int_{\D} L^{R}[\gn^{l,-}](t,v,z)\fe(t,v,z)\gn^{l,-}(t,v,z) \,\rd v\\[4pt]
\leq & {\color{black}{C_{R,L,d}^{-}(B)}}\,  \| \gn^{l}(t,\cdot,z) \|_{L_{v}^1(\D)} \| \fe^{-}(t,\cdot,z) \|_{L_{v}^2(\D)} \| \gn^{l,-}(t,\cdot,z) \|_{L_{v}^2(\D)} \\[6pt]
& + {\color{black}{C_{R,L,d}^{-}(B)}} \, \| \gn^{l,-}(t,\cdot,z) \|_{L_{v}^1(\D)} \| \fe(t,\cdot,z) \|_{L_{v}^2(\D)} \| \gn^{l,-}(t,\cdot,z) \|_{L_{v}^2(\D)} \\[6pt]
\leq & {\color{black}{C_{R,L,d}^{-}(B)}} \, \| \gn^{l}(t,\cdot,z) \|_{L_{v}^{1}(\D)}
\left| \fe^{-}(t,\cdot,z) \right|^{2}_{L_{v}^2,r} + {\color{black}{\tilde C_{R,L,d}^{-}(B)}} \, \| \fe(t,\cdot,z) \|_{L_{v}^2(\D)} \left| \fe^{-}(t,\cdot,z) \right|^{2}_{L_{v}^2,r}\,, 
    \end{split}
\end{equation}
}
where the estimate for the loss term \eqref{QLLp} is used. 

To estimate the third term in Eq.~\eqref{QR_Z_split}, $Q^R(\fe, \gn^{l}) = Q^{R,+}(\fe, \gn^{l}) - Q^{R,-}(\fe, \gn^{l})$, the calculation is almost the same as above and we omit the details here. 
One gets
\begin{equation}\label{QR_plus2}
\begin{split}
&\int_{\D} Q^{R,+}(\fe, \gn^{l})(t,v,z) \gn^{l}\mathbf{1}_{\left\lbrace \gn^{l} 
\leq 0\right\rbrace}(t,v,z) \,\rd v\\  
\leq &\ \int_{\D}\left[ Q^{R,+}(\fe^{+}, \gn^{l,-}) + Q^{R,+}(\fe^{-},\gn^{l,+}) \right](t,v,z)\, \gn^{l,-}(t,v,z) \,\rd v \\[4pt]
\leq &\ C_{R,L,d,2}^{+}(B)\, \| \fe(t,\cdot,z) \|_{L_{v}^1} \| \gn^{l,-}(t,\cdot,z) \|_{L^2(\D)}^2 \\[4pt]
&\qquad + 
{\color{black}{C_{R,L,d,2}^{+}(B)}}\, \| \fe^{-}(t,\cdot,z) \|_{L_{v}^1(\D)} \| \gn^{l}(t,\cdot,z) \|_{L_{v}^2(\D)} \| \gn^{l,-}(t,\cdot,z) \|_{L_{v}^2(\D)} \\[4pt]
\leq & {\color{black}{C_{R,L,d,2}^{+}(B)}}\, \| \fe(t,\cdot,z) \|_{L_{v}^1} | \fe^{-}(t,\cdot,z) |^{2}_{L_{v}^2,r} + {\color{black}{\tilde C_{R,L,d,2}^{+}(B)}}\, \| \gn^{l}(t,\cdot,z) \|_{L_{v}^2(\D)} | \fe^{-}(t,\cdot,z) |^{2}_{L_{v}^2,r}\,. 
\end{split}
\end{equation}  
Moreover,
\begin{multline}\label{QR_neg2}
    - \int_{\D} Q^{R,-}(\fe, \gn^{l})(t,v,z)  \gn^{l}\mathbf{1}_{\left\lbrace \gn^{l}\leq 0\right\rbrace}(t,v,z) \,\rd v \\
	 \leq {\color{black}{C_{R,L,d}^{-}(B)}}\, \| f_N(t,\cdot,z) \|_{L_{v}^1(\D)} \| \gn^{l,-}(t,\cdot,z) \|^{2}_{L_{v}^2(\D)}\,. 
\end{multline}

Combining \eqref{gn_char}, \eqref{EN}, \eqref{QR_plus}--\eqref{QR_neg2}, taking $\sup_{z\in I_z}$ and adding up all $|l| \leq r$, for any $t \in [0,t_0]$, one has
\small{
\begin{equation}
\begin{split}
		\frac{\rd}{\rd t} \left\| \fe^{-}(t,\cdot,z) \right\|_{L_{v}^2,r} 
		& \leq C_{R,L,d,2}^{+,-}(B) \left( \| \fe(t,\cdot,\cdot)\|_{L_{v}^1,r} + \| \fe(t,\cdot,\cdot) \|_{L_{v}^2,r} \right) 
		 \left\| \fe^{-}(t,\cdot,z) \right\|_{L_{v}^2,r}\\[4pt]
		& \qquad + \frac{C_{\mP}C_{R,L,d,1,r}(B)\, K^2_{1,r}(t_0)}{N}\\[4pt]
		& \leq C_{R,L,d,2}^{+,-}(B) \left( E_r + K_{0,r}(t_0)  \right) \left\| \fe^{-}(t,\cdot,z) \right\|_{L_{v}^2,r} + \frac{C_{\mP}C_{R,L,d,1,r}(B)\, K^2_{1,r}(t_0)}{N} \\[4pt]
		& =: \mathcal{C}_{K_{0,r}} \left\| \fe^{-}(t,\cdot,z) \right\|_{L_{v}^2,r}
		+ \frac{\mathcal{C}_{K_{1,r}}}{N}\,, 
	\end{split}
\end{equation} 
}
where $C_{R,L,d,2}^{+,-}(B)$ represents the constant depending on both $C_{R,L,d,2}^{+}(B)$ and $C_{R,L,d}^{-}(B)$. 

\end{proof}


\subsection{Proof of the well-posedness of \texorpdfstring{$f_N$}{fN} with uncertainties}
\label{subsec:wellposed}

\subsubsection{Local well-posedness of \texorpdfstring{$f_N$}{fN} with uncertainties on a small time interval $[t_0, t_0+\tau]$}
\label{subsub:local}

In this subsection, we will show the local well-posedness of the numerical solution $f_N$ by constructing a space $\chi$ where all the $z$-derivatives are included.

\begin{proposition}
\label{localexistence}
Let the truncation parameters $R$, $L$ satisfy \eqref{RL} and the collision kernel $B$ satisfy \eqref{kernel}--\eqref{cutoff} and \eqref{kinetic}--\eqref{Assump_B}. If we further assume that the initial condition $f^0(v,z)$ to the original problem \eqref{ABE} belongs to $L_{v}^{1},r\cap L_v^2,r$ with the following quantities
	\begin{equation}
	E^{f^0}_{r}=\|f^0(\cdot,\cdot)\|_{L^1_v,r},  \quad  D^{f^0}_{0,r}=\left\|f^{0}(\cdot,\cdot)\right\|_{L^2_v,r},
	\end{equation}
	and that the numerical system \eqref{PFS} is evolved from a certain time $t_0$ with
	\begin{equation}\label{fN_IC}
	\|f_N(t_0,\cdot,\cdot)\|_{L^1_v,r} \leq 2 E^{f^0}_{r}, \quad \|f_N(t_0,\cdot,\cdot)\|_{L^2_v,r} \leq K_{0,r},
	\end{equation}
	then there exists a local time $ \tau$ such that \eqref{PFS} admits a unique solution $ f_{N}= f_{N}(t,v,z) \in L_{v}^{1},r \cap L_v^2,r $ on $ [t_0,t_0+\tau]$. In particular, one can choose
	\begin{equation} \label{tau}
		\tau=\frac{1}{2(\bar{C}_1 \bar{D}+\bar{C}_2\bar{E})}, \quad \text{with} \quad \bar{E} = 4E^{f^0}_{r},  \quad \bar{D} = 2K_{0,r},
	\end{equation}
	such that
	\begin{equation}
		\forall t\in [t_0,t_0+\tau], \quad  \|f_N(t,\cdot,\cdot)\|_{L^1_v,r} \leq \bar{E}, \quad  \|f_N(t,\cdot,\cdot)\|_{L^2_v,r} \leq \bar{D},
	\end{equation}
	where $T$ is the final prescribed time and the constants $\bar{C}_1$, $\bar{C}_2$ only depend on the truncation parameters $R$, $L$, dimension $d$, and the collision kernel $B$.
\end{proposition}
\begin{proof}	
	The main strategy is to use the Banach fixed point theorem to prove that the numerical system \eqref{PFS} admits a solution $f_N$ in our designated space $\chi$ defined based on our newly-established norm $\|\cdot\|_{L^{1}_{v},r}$ and $\|\cdot\|_{L^{2}_{v},r}$: with undermined constants $ \bar{E}, \bar{D}> 0$ and small enough time $ \tau >0 $ that will be specified later,
	\begin{equation}
	\begin{split}
	    \chi= \Big\lbrace  f\in L^{\infty}([t_0,t_0+\tau]; L^{1}_{v},r \cap L^{2}_{v},r): &\sup\limits_{t\in [t_0,t_0+\tau]}\left\|f(t,\cdot,\cdot)\right\|_{L^{1}_{v},r} \leq \bar{E},\\
     &\sup\limits_{t\in [t_0,t_0+\tau]}\left\|f(t,\cdot,\cdot)\right\|_{L^{2}_{v},r} \leq \bar{D} \Big\rbrace,
	\end{split}
	\end{equation}
	which is a complete metric space with respect to the induced distance: 
	\begin{equation}\label{distance}
		d(f, \tilde{f}) : = \left\| f - \tilde{f} \right\|_{\chi} =\sup\limits_{t\in [t_0,t_0+\tau]}  \left\| f(t,\cdot,\cdot) - \tilde{f}(t,\cdot,\cdot) \right\|_{L^{2}_{v},r}.
	\end{equation}
	For any $f_{N} \in \chi$, we can define the operator $ \Phi[f_{N}] $ for any $t\in[t_0,t_0+\tau]$,
	\begin{equation}
		\Phi[f_{N}](t,v,z) = f_{N}(t_0,v,z) + \int_{t_0}^{t}\mP_N Q^R(\fe,\fe)(s,v,z) \,\rd s.
	\end{equation}
	
	We proceed to show that there exists a unique fixed point in $\chi$ for the mapping $\Phi$. 
	
	Step (I): We first show that $ \Phi $ maps $ \chi$ into itself: $\Phi[f_{N}] \in \chi$, i.e., for any $ f_{N} \in \chi $ and $t\in [t_0,t_0+\tau]$,
	\begin{equation*} 
		\begin{split}
			&\left\| \Phi[f_{N}](t,\cdot,\cdot)\right\|_{L^{1}_{v},r}\\[4pt]
            \leq & \left\| f_{N}(t_0,\cdot,\cdot) \right\|_{L^{1}_{v},r} + \int_{t_0}^{t} \left\|  \mP_N Q^{R}(f_{N},f_{N})(s,\cdot,\cdot)\right\|_{L^{1}_{v},r} \,\rd s \\[4pt]
			\leq & \left\| f_{N}(t_0,\cdot,\cdot) \right\|_{L^{1}_{v},r}  + \tau (2L)^{d/2} \sup\limits_{t\in [t_0,t_0+\tau]} \left\|\mP_N Q^R(f_{N},f_{N})(t,\cdot,\cdot)\right\|_{L^{2}_{v},r} \\[4pt]
			\leq & \left\| f_{N}(t_0,\cdot,\cdot)  \right\|_{L^{1}_{v},r} + \tau C''_{R,L,d,2,r}(B) (2L)^{d/2} \sup\limits_{t\in [t_0,t_0+\tau]}\left( \left\| f_{N}(t,\cdot,\cdot) \right\|_{L^{1}_{v},r} \left\| f_{N}(t,\cdot,\cdot) \right\|_{L^{2}_{v},r}\right) \\[4pt]
			\leq & \left\|f_{N}(t_0,\cdot, \cdot)  \right\|_{L^{1}_{v},r}  + \tau C''_{R,L,d,2,r}(B) (2L)^{d/2} \bar{E} \bar{D}, 
		\end{split}
	\end{equation*}
	where the estimate \eqref{QHKZ||} is used in third inequality above. Similarly,
	\begin{equation*} 
		\begin{split}
			&\left\| \Phi[f_{N}](t,\cdot,\cdot)\right\|_{L^{2}_{v},r} \\[4pt]
            \leq & \left\| f_{N}(t_0,\cdot,\cdot) \right\|_{L^{2}_{v},r} + \int_{t_0}^{t} \left\|  \mP_N Q^{R}(f_{N},f_{N})(s,\cdot,\cdot)\right\|_{L^{2}_{v},r} \,\rd s \\[4pt]
			\leq & \left\| f_{N}(t_0,\cdot,\cdot) \right\|_{L^{2}_{v},r}  + \tau \sup\limits_{t\in [t_0,t_0+\tau]} \left\|\mP_N Q^R(f_{N},f_{N})(t,\cdot,\cdot)\right\|_{L^{2}_{v},r} \\[4pt]
			\leq & \left\| f_{N}(t_0,\cdot,\cdot) \right\|_{L^{2}_{v},r} +  \tau C''_{R,L,d,2,r}(B)  \sup\limits_{t\in [t_0,t_0+\tau]}\left( \left\| f_{N}(t,\cdot,\cdot) \right\|_{L^{1}_{v},r} \left\| f_{N}(t,\cdot,\cdot) \right\|_{L^{2}_{v},r}\right) \\[4pt]
			\leq & \left\|f_{N}(t_0,\cdot,\cdot)  \right\|_{L^{2}_{v},r}  + \tau C''_{R,L,d,2,r}(B) \bar{E} \bar{D}, 
		\end{split}
	\end{equation*}
	
	Step (II): We next show that $ \Phi $ is a contraction mapping on $ \chi $. For any $f_{N}, \tilde{f}_{N} \in\chi$ with the same initial datum $ f_{N}(t_0,v,z)$, we have
	\begin{equation*} 
		\begin{split}
			\left\|\Phi[f_{N}]-\Phi[\tilde{f}_{N}] \right\|_{\chi} =& \sup\limits_{t\in [t_0,t_0+\tau]} \left\|\Phi[f_{N}](t,\cdot,\cdot)-\Phi[\tilde{f}_{N}](t,\cdot,\cdot)\right\|_{L^{2}_{v},r}\\[4pt]
			\leq & \sup\limits_{t\in [t_0,t_0+\tau]}\int_{t_0}^{t}  \left\|\mP_N Q^R(f_{N},f_{N})(s,\cdot, \cdot)-\mP_N Q^R(\tilde{f}_{N},\tilde{f}_{N})(s, \cdot, \cdot)\right\|_{L^{2}_{v},r} \,\rd s \\[4pt]
			\leq & \tau \sup\limits_{t\in [t_0,t_0+\tau]}  \left\|Q^R(f_{N},f_{N})(t, \cdot,\cdot)-Q^R(\tilde{f}_{N},\tilde{f}_{N})(t,\cdot,\cdot)\right\|_{L^{2}_{v},r} \\[4pt]
			\leq & \tau \sup\limits_{t\in [t_0,t_0+\tau]} \left( \left\|Q^R(f_{N}-\tilde{f}_{N},f_{N})(t, \cdot,\cdot)\right\|_{L^{2}_{v},r} + \left\|Q^R(\tilde{f}_{N},f_{N}-\tilde{f}_{N})(t,\cdot, \cdot)\right\|_{L^{2}_{v},r}\right )\\[4pt]
			\leq & \tau C''_{R,L,d,2,r}(B) \sup\limits_{t\in [t_0,t_0+\tau]} \Big( \left\|f_{N}(t,\cdot,\cdot) -\tilde{f}_{N}(t,\cdot,\cdot)\right\|_{L^{1}_{v},r} \|f_{N}(t,\cdot,\cdot)\|_{L^{2}_{v},r}\\[4pt]
			&\qquad \qquad \qquad \qquad \qquad \quad  + \left\|f_{N}(t,\cdot,\cdot)-\tilde{f}_{N}(t,\cdot,\cdot) \right\|_{L^{2}_{v},r} \|\tilde{f}_{N}(t,\cdot,\cdot)\|_{L^{1}_{v},r} \Big)\\[4pt]
			\leq & \tau C''_{R,L,d,2,r}(B) ((2L)^{d/2}\bar{D}+\bar{E}) \left(\sup\limits_{t\in [t_0,t_0+\tau]}  \left\|f_{N}(t,\cdot,\cdot)-\tilde{f}_{N}(t,\cdot, \cdot)\right\|_{L^{2}_{v},r}\right)\\[4pt]
			=&  \tau \left(C''_{R,L,d,2,r}(B)(2L)^{d/2}\bar{D} + C''_{R,L,d,2,r}(B)\bar{E}\right)\left\|f_{N}-\tilde{f}_{N}\right\|_{\chi}.
		\end{split}
	\end{equation*}
	Therefore, in order to apply the Banach fixed point theorem, we have to prove that the operator $\Phi: \chi \rightarrow \chi$ is a contraction mapping, in the sense that, if we set $\bar{C}_1=C''_{R,L,d,2,r}(B)(2L)^{d/2}$, $\bar{C}_2=C''_{R,L,d,2,r}(B)$,
    \begin{equation}
    \begin{cases}
		\left\| f_{N}(t_0,\cdot,\cdot) \right\|_{L^1_v,r} + \tau \bar{C}_1 \bar{E} \bar{D} \leq \bar{E}, \\[4pt]
		\left\|f_{N}(t_0,\cdot,\cdot) \right\|_{L^2_v,r} + \tau \bar{C}_2 \bar{E} \bar{D} \leq \bar{D},\\[4pt]
        \tau(\bar{C}_1 \bar{D}+\bar{C}_2\bar{E}) < 1,
	\end{cases}
    \end{equation}
	which can actually be verified if we choose $\bar{D}$, $\bar{E}$ and $\tau$ as given in \eqref{tau}. Thus, there exists a unique solution on $[t_0,t_0+\tau]$ to the numerical system \eqref{PFS}.
\end{proof}

\begin{remark}
In addition to proving the local-wellposedness in a certain space where all the $z$-derivatives are considered, we can also provide a new iteration scheme to show the local well-posedness of each $g^{l}_N(t,v,z) = \partial_z^l f_N(t,v,z) $ by taking $l$-th order $z$-derivative on both sides of $\partial_t f_N = \mathcal{P}_N Q^R(f_N, f_N)$, i.e., $g^{l}_N = \partial_z^l f_N $ satisfies the numerical system:
\begin{equation}
   \partial_t \partial_z^l f_N(t,v,z) = \mathcal{P}_N \partial_z^l Q^R(f_N,f_N)(t,v,z). 
\end{equation}
for fixed $z\in I_z$. 
\end{remark}

\subsubsection{Well-posedness of \texorpdfstring{$f_N$}{fN} with uncertainties on an arbitrary bounded time interval \texorpdfstring{$[0,T]$}{[0,T]}}
\label{subsub:main}

In this subsection, we are now ready to present our main result about the well-posedness of the numerical solution $f_{N}$ with uncertainties within any arbitrarily prescribed time interval.

\begin{theorem}\label{existencetheorem}
	Let the truncation parameters $R$, $L$ satisfy \eqref{RL} and the collision kernel $B$ satisfy \eqref{kernel}--\eqref{cutoff} and  \eqref{kinetic}--\eqref{Assump_B}.
	If the initial condition $f^{0}(v,z)$ in \eqref{ABE} and the initial approximation $f_N^0(v,z)$ in \eqref{PFS} satisfy the assumptions specified in Section~\ref{subsec:initial}.
	
	Then there exists an integer $N_0$ depending on the prescribed final time $T$ and initial condition $f^0$, such that for all $ N>N_{0} $, the numerical system \eqref{PFS} admits a unique solution $ \fe = \fe(t,\cdot,\cdot) \in L^1_v,r \cap H^1_v,r $ on the time interval $ [0,T] $ with the following estimates:
	\begin{equation}\label{fNL1L2}
	\forall t\in [0,T], \quad \left\| f_{N}(t,\cdot,\cdot) \right\|_{L^1_v,r} \leq 2E^{f^0}_r, \quad \left\| f_{N}(t,\cdot,\cdot) \right\|_{L^2_v,r} \leq K_{0,r}(T),
	\end{equation}
	where $E^{f^0}_{r}=\|f^0(\cdot,\cdot)\|_{L^{1}_{v},r}$ and $K_{0,r}(T)$ is the constant depending on prescribed time $T$, the truncation parameters $R,L$, dimension $d$, the collision kernel $B$ and initial quantity $D^{f^0}_{0,r}=\|f^0(\cdot,\cdot)\|_{L^{2}_{v},r}$.
\end{theorem}

\begin{proof}
The key part of the proof is to extend the local well-posedness result in Proposition \ref{localexistence} to any arbitrarily prescribed time interval $[0,T]$ by time iteration.

Step (I): We start with the initial time $t=0$, by using condition \eqref{con(c)}, we are able to choose $N_1$ such that for any $N \geq N_1$,
	\begin{equation} \label{initial}
	\|f^0_N(\cdot,\cdot)\|_{L^{1}_{v},r} \leq 2E^{f^0}_r.
	\end{equation}
Moreover, we have $\|f^0_N(\cdot,\cdot)\|_{L^{2}_{v},r} \leq D^{f^0}_{0,r} \leq K_{0,r}(T)$ due to the condition \eqref{con(b)} and $K_{0,r}(T)$ can be explicitly given by prescribed time $T$, the truncation parameters $R,L$, collision kernel $B$ and initial quantity $D^{f^0}_{0,r}$ by noticing Corollary \ref{regularity1}. 

Then, by applying Proposition~\ref{localexistence}, there exists a unique solution $f_N(t,\cdot,\cdot)\in L^{1}_{v},r \cap L^{2}_{v},r$ over the local time interval $[0,\tau]$ with the following estimate,
\begin{equation}
\forall t\in [0,\tau], \quad \|f_N(t,\cdot,\cdot)\|_{L^{1}_{v},r}\leq 4E^{f^0}_r.
\end{equation}
Furthermore, by taking advantage of the boundedness in $L^{1}_{v},r$ and that $f_N^0(\cdot,\cdot)\in {H^1_v,r}$ from \eqref{con(b)}, we can invoke Corollary~\ref{regularity1} to find the $L^{2}_{v},r$- and $H^1_v,r$- estimates of the numerical solution $f_N$ in the local time interval
\begin{equation}
\forall t\in [0, \tau], \quad \|f_N(t,\cdot,\cdot)\|_{L^{2}_{v},r}\leq K_{0,r}(\tau), \quad \|f_N(t,\cdot,\cdot)\|_{H^1_v,r} \leq K_{1,r}(\tau),
\end{equation}
as well as the estimate of the negative part $f^{}_{N}$ in $L^{2}_{v},r$,
\begin{equation}\label{fN-}
\begin{split}
\forall t\in [0, \tau], \quad \left\| \fe^{-}(t,\cdot, \cdot) \right\|_{L^2_{v},r} \leq e^{\mathcal{C}_{K_{0,r}}\, \tau} \left( \| \fe^{0,-}(\cdot, \cdot) \|_{L^2_{v},r} + \frac{\mathcal{C}_{K_{1,r}}}{N} \right).
\end{split}
\end{equation}

On the other hand, noticing that $|\partial_z^{l}f_N(t,v,z)| = 2\partial_z^{l}f_N^-(t,v,z) + \partial_z^{l}f_N(t,v,z)$, we have
\begin{equation}\label{nL1}
	\begin{split}
		\|f_N(t,\cdot,\cdot)\|_{L^{1}_{v},r}= &\sum_{|l|\leq r} \sup_{z \in I_z} \int_{\D}|\partial_z^{l} f_N(t,v,z)|\,\rd{v}\\[4pt]
		=&2 \sum_{|l|\leq r} \left(\sup_{z \in I_z}\int_{\D}\partial_z^{l} f_N^-(t,v,z)\,\rd{v}+\sup_{z \in I_z}\int_{\D}\partial_z^{l}f_N(t,v,z)\,\rd{v} \right)\\[4pt]
		=&2 \sum_{|l|\leq r} \sup_{z \in I_z} \|\partial_z^{l} f_N^-(t,\cdot,z)\|_{L^{1}_{v}}+ \sum_{|l|\leq r} \sup_{z \in I_z} \int_{\D}\partial_z^{l} f^0(v,z)\,\rd{v}\\[4pt]
		\leq& 2(2L)^{d/2} \|f_N^-(t,\cdot,\cdot)\|_{L^{2}_{v},r}+ E^{f^0}_{r},
	\end{split}
\end{equation}
where we used the important mass conservation property in Lemma~\ref{lemma:conv} and the assumption \eqref{con(a)}. 

Therefore, observing the estimate \eqref{nL1}, it implies that $\|f_N(t,\cdot,\cdot)\|_{L^1_v,r}$ will be under control, if we have a "good" estimate for the negative part $\|f_N^-(t,\cdot,\cdot)\|_{L^2_v,r}$.
Then, thanks to the estimate \eqref{fN-}, we can simply choose $N_2$ large enough such that for all $N \geq N_2$ we can define the quantity $\bar{K}$ to satisfy
\begin{equation}\label{N0}
\begin{split}
\bar{K}:= e^{ \mathcal{C}_{K_{0,r}}\, T} \left( \| \fe^{0,-}(\cdot, \cdot) \|_{L^2_{v},r} + \frac{\mathcal{C}_{K_{1,r}}}{N} \right) \leq \frac{E^{f^0}_r}{2(2L)^{d/2}},
\end{split}
\end{equation}
where the inequality above is possible to achieve thanks to the condition \eqref{con(d)} and always holds for any 
$t_0\leq T$, as the quantity $\bar{K}$ is an increasing function with respect to time.

Hence, by combining \eqref{fN-}-\eqref{nL1}, we find that
\begin{equation} \label{fNL1L21}
\forall t\in [0,\tau], \quad \|f_N(t,\cdot,\cdot)\|_{L^{1}_{v},r}\leq 2E^{f^0}_r.
\end{equation}

By choosing $N_0$ as the maximum of among the $N_1$ determined to satisfy \eqref{initial} and $N_2$ to satisfy \eqref{N0}, we have found such an integer $N_0$, depending only on the prescribed final time $T$ and initial condition $f^0$, that for all $N > N_0$, the numerical system \eqref{PFS} admits a unique solution $f_N(t,\cdot,\cdot)\in L^{2}_{v},r \cap H^1_v,r$ in the time interval $[0,\tau]$, which satisfies \eqref{fNL1L21}. 

Step (II): By letting $t=\tau$ as a new initial time, we need to check if the local well-posedness can be extended in the equally-long time interval $[\tau,2\tau]$. Recalling Step (I), we have already proved that
\begin{equation}\label{initial1}
\forall t\in [0,\tau], \quad f_N(t,\cdot,\cdot)\in L^{1}_{v},r \cap H^1_v,r \quad \text{and}\quad \|f_N(t,\cdot,\cdot)\|_{L^{1}_{v},r}\leq 2E^{f^0}_r.
\end{equation}
Then by taking $k=0$ in the Proposition~\ref{regularity}, we have the following estimate for the propagation of $\|f_N(t,\cdot,\cdot)\|_{L^{2}_{v},r}$,
\begin{equation}\label{initial2}
\begin{split}
    \forall t \in [0,\tau],  \quad \|f_N(t,\cdot,\cdot)\|_{L^{2}_{v},r} \leq K_{0,r}(\tau) \leq K_{0,r}(T)
\end{split}
\end{equation}
From the estimate \eqref{initial1} - \eqref{initial2}, we find that $\|f_N(\tau,\cdot,\cdot)\|_{L^{1}_{v},r}$ and $\|f_N(\tau,\cdot,\cdot)\|_{L^{2}_{v},r}$ satisfy the condition of the local well-posedness Proposition~\ref{localexistence}, which allows us to apply Proposition~\ref{localexistence} starting from $t=\tau$ and further obtain that 
there exists a unique solution $f_N(t,\cdot,\cdot)\in L^{1}_{v},r \cap L^{2}_{v},r$ on $[\tau, 2\tau]$ with
\begin{equation}
\forall t\in [\tau, 2\tau], \quad \|f_N(t,\cdot,\cdot)\|_{L^{1}_{v},r}\leq 4E^{f^0}_r.
\end{equation}

Meanwhile, noticing the bounded property in $L^{1}_{v},r$ above and the fact that $f_N^0(\cdot,\cdot) \in {H^1_v,r}$, we can invoke the Corollary~\ref{regularity1} over the interval $[0,2\tau]$ to derive that
\begin{equation}
\forall t\in [0, 2\tau], \quad \|f_N(t,\cdot,\cdot)\|_{L^{2}_{v},r}\leq K_{0,r}(2\tau), \quad \|f_N(t,\cdot,\cdot)\|_{H^1_v,r}\leq K_{1,r}(2\tau),
\end{equation}
and for any $t\in [0,2\tau]$, 
\begin{equation} 
\begin{split}
		 \left\|\fe^{-}(t,\cdot,\cdot)\right\|_{L^{2}_v,r}  \leq \e^{ \mathcal{C}_{K_{0,r}}\, (2\tau)} \left( \| \fe^{0,-}(\cdot, \cdot) \|_{L^2_{v},r} + \frac{\mathcal{C}_{K_{1,r}}}{N} \right) \leq \bar{K}
\end{split}
\end{equation}
i.e., the same choice of $N$ chosen above would still make
\begin{equation} 
\forall t\in [0,2\tau], \quad \|f_N(t,\cdot,\cdot)\|_{L^{1}_{v},r}\leq 2E^{f^0}_r.
\end{equation}
That is, at time point $t=2\tau$, we are back to the situation \eqref{initial1} at $t=\tau$. In fact, we can generalize the same strategy to longer time interval $[0,n\tau]$ with the same choice $N\geq N_0$ and $\bar{K}$, in the sense that, for all $t\in [0, n\tau]$,
\begin{equation}
    \begin{split}
        \|f_N(t,\cdot,\cdot)\|_{L^{2}_{v},r}\leq &  K_{0,r}(n\tau), \quad \|f_N(t,\cdot,\cdot)\|_{H^1_v,r}\leq K_{1,r}(n\tau),\\
        \left\|\fe^{-}(t,\cdot,\cdot)\right\|_{L^{2}_v,r} \leq & \bar{K}, \qquad\qquad \|f_N(t,\cdot,\cdot)\|_{L^{1}_{v},r}\leq 2E^{f^0}_r.
    \end{split}
\end{equation}

Step (III): Repeating Step (II) for $n$ times until the time interval $[0,n\tau]$ covers the prescribed interval $[0,T]$, we can show that there exists a unique solution $f_N(t,\cdot,\cdot)\in L^{1}_{v},r \cap H^1_v,r$ on $[0,T]$ with
\begin{equation}
\forall t\in [0,T], \quad \|f_N(t,\cdot,\cdot)\|_{L^1_v,r}\leq 2E^{f^0}_r, \quad \left\| f_{N}(t,\cdot,\cdot) \right\|_{L^{2}_v,r} \leq K_{0,r}(T),
\end{equation}
where the estimate of $\left\| f_{N}(t,\cdot,\cdot) \right\|_{L^{2}_v,r}$ can be verified by taking $k=0$ in Proposition~\ref{regularity} once again.
\end{proof}

\section{Convergence of the Fourier-Galerkin spectral method for Boltzmann equation with uncertainties}
\label{sec:conv}

In this section, we will prove the convergence of the proposed spectral method by taking advantage of the well-posedness and stability of the numerical solution $f_N$ established in the previous section.

For the continuous system \eqref{ABE} with a periodic, non-negative (in $v$) initial condition $f^0(v,z) $ in $ L^{1}_{v},r \cap H^{k}_{v},r$ for some integer $k \geq 1$, 
there exists a unique global non-negative solution $f(t,v,z)\in H^{k}_{v},r$ with the estimate that $\|f(t,\cdot,\cdot)\|_{H^{k}_{v},r}\leq C_{k,r}$ for all $t \geq 0$ with $C_{k,r}$ a constant depending on initial datum $f^0$, which can be shown by following the similar argument to handle the deterministic model as in \cite[Proposition 5.1]{FM11} coupled with our estimate of the collision operator including uncertainties in Corollary \ref{sumQz}.

For the numerical system \eqref{PFS}, we consider the initial condition $f_N(0,v,z)=\mP_N f^{0}(v,z)$, which satisfies the four conditions \eqref{con(a)}--\eqref{con(d)}(e.g., $f^{0}(v,z)$ is a H\"{o}lder continuous in $v$).
Then, by Theorem~\ref{existencetheorem}, there exists a unique solution $f_N(t,\cdot,\cdot) \in L^{1}_{v},r \cap H^{k}_{v},r$ over the whole time interval $[0,T]$ with $\|f_N(t,\cdot,\cdot)\|_{L^{2}_{v},r}\leq K_{0,r}(T)$ for all $t \in[0,T]$.

Define the error function $e_{N}$ with uncertainties
\begin{equation}
e_{N}(t,v,z) := \mathcal{P}_{N} f(t,v,z) - f_{N}(t,v,z).
\end{equation}
We can show the following main theorem of the convergence,
\begin{theorem}\label{spectralaccuracy}
Let the truncation parameters $R$, $L$ satisfy \eqref{RL} and the collision kernel $B$ satisfy \eqref{kernel}--\eqref{cutoff} and  \eqref{kinetic}--\eqref{Assump_B}. For a periodic, non-negative (in $v$) initial condition $f^0(v,z) $ in $ L^{1}_{v},r \cap H^{k}_{v},r$ for some integer $k \geq 1$, choose $N_0$ to satisfy the condition in Theorem~\ref{existencetheorem}, then the numerical system \eqref{PFS} by Fourier-Galerkin spectral method is convergent for all $N>N_0$ and exhibits spectral accuracy in the sense that
\begin{equation}\label{final}
\forall t\in [0,T], \quad \left\| e_{N}(t,\cdot,\cdot) \right\|_{L^{2}_{v},r} \leq C_{r}(T,f^0)\left(\frac{1}{N^{k}} \right), \quad \text{for all } N>N_0,
\end{equation}
where $C_{r}(T,f^0)$ is a constant depending on the truncation parameters $R,L$, dimension $d$,  collision kernel $B$, parameter $r$, prescribed time $T$, and initial condition $f^0(v,z)$.

\end{theorem}
\begin{proof}
	To obtain the differential equation satisfied by our defined error function $e_N$, we first need to apply the projection operator $\mP_N$ to the original problem \eqref{ABE}:
	\begin{equation}\label{PUP}
	\left\{
	\begin{array}{lr}  
	\partial_{t} \mathcal{P}_{N}f(t,v,z) = \mathcal{P}_{N}Q^{R}(f,f)(t,v,z),\\[4pt]
	f_{N}(0,v,z) = \mathcal{P}_{N}f(0,v,z), 
	\end{array}
	\right.
	\end{equation}
	Then, by taking subtraction between \eqref{PFS} and \eqref{PUP}, we find that
	\begin{equation}\label{error1}
	\left\{
	\begin{aligned}
	&\partial_{t} e_{N}(t,v,z)= \mathcal{P}_{N} Q^{R}(f,f)(t,v,z) - \mathcal{P}_{N}Q^{R}(f_{N},f_{N})(t,v,z),\\
	& e_{N}(0,v,z) = 0.
	\end{aligned}
	\right.
	\end{equation}
	where the zero initial error comes with the fact that $f_{N}(0,v,z) = \mathcal{P}_{N}f(0,v,z) $ after applying the projection.
	Next, we take $\partial^l_z$ to both hand sides of \eqref{error1} and multiply by $ \partial^l_z e_{N}$, then after integrating over $\D$, taking the supremum over $I_{z}$ and adding up all $|l|\leq r$ altogether, we have
	\begin{equation}
	\begin{split}
	\frac{1}{2} \frac{\rd}{\rd t} \left\| e_N(t,\cdot,\cdot) \right\|^{2}_{L^{2}_{v},r}  =& \sum_{|l|\leq r}\sup_{z\in I_{z}}\int_{\D} \partial^l_z\left[ \mathcal{P}_{N} Q^{R}(f,f) - \mathcal{P}_{N}Q^{R}(f_{N},f_{N})\right](t,v,z) \partial^l_z e_{N}(t,v,z) \,\rd v \\[4pt]
	\leq & C_{r} \left\| \mathcal{P}_{N}\left( Q^{R}(f,f) - Q^{R}(f_{N},f_{N}) \right)(t,\cdot,\cdot)\right\|_{L^{2}_{v},r} \left\|  e_{N}(t,\cdot,\cdot) \right\|_{L^{2}_{v},r}\\[4pt]
    \leq &C_{r}C_{\mP} \left\| \left( Q^{R}(f,f) - Q^{R}(f_{N},f_{N}) \right)(t,\cdot,\cdot)\right\|_{L^{2}_{v},r} \left\|  e_{N}(t,\cdot,\cdot) \right\|_{L^{2}_{v},r}
	\end{split}
	\end{equation}
	which implies that
	\begin{equation}\label{II1}
	\frac{\rd }{\rd t} \left\| e_N(t,\cdot,\cdot) \right\|_{L^{2}_{v},r} \leq C'_{r}\left\| Q^{R}(f,f)(t,\cdot,\cdot) - Q^{R}(f_{N},f_{N})(t,\cdot,\cdot) \right\|_{L^{2}_{v},r}. 
	\end{equation}
	By estimating the right-hand side of the inequality \eqref{II1} above,
	\begin{equation}
	\begin{split}
	&\left\| Q^{R}(f,f)(t,\cdot,\cdot) - Q^{R}(f_{N},f_{N})(t,\cdot,\cdot) \right\|_{L^{2}_{v},r}\\[4pt]
	\leq& \left\|  Q^{R}(f-f_{N},f)(t,\cdot,\cdot) \right\|_{L^{2}_{v},r} + \left\|  Q^{R}(f_{N},f-f_{N})(t,\cdot,\cdot) \right\|_{L^{2}_{v},r}\\[4pt]
	\leq  &C_{R,L,d,2,r}^{''}(B) \left\|  f(t,\cdot,\cdot) - f_{N}(t,\cdot,\cdot) \right\|_{L^{2}_{v},r} \left(\left\|  f(t,\cdot,\cdot)  \right\|_{L^{2}_{v},r} + \left\| f_{N}(t,\cdot,\cdot) \right\|_{L^{2}_{v},r} \right)\\[4pt]
	\leq  &C_{R,L,d,2,r}^{''}(B)\left( C_{0,r}+ K_{0,r}\right) \left\|  f(t,\cdot,\cdot) - f_{N}(t,\cdot,\cdot) \right\|_{L^{2}_{v},r}.
	\end{split}
	\end{equation}
	where the bi-linearity of collision operator $Q^R$ and estimate \eqref{QHKZ||} are utilized in the second inequality, while the well-posedness theorem, as well as the associated estimate of the theoretical solution $f$ (i.e., $C_{0,r}$) and numerical counterpart $f_N$ (i.e., $K_{0,r}$), are applied in the last inequality above. Furthermore, we also have,
	\begin{equation}
	\begin{split}
	\left\|  f(t,\cdot,\cdot) - f_{N}(t,\cdot,\cdot) \right\|_{L^{2}_{v},r} 
	=& \left\|  f(t,\cdot,\cdot) - \mathcal{P}_{N}f(t,\cdot,\cdot) + \mathcal{P}_{N}f(t,\cdot,\cdot) - f_{N}(t,\cdot,\cdot) \right\|_{L^{2}_{v},r} \\[4pt]
	\leq & \left\|  f(t,\cdot,\cdot) - \mathcal{P}_{N}f(t,\cdot,\cdot) \right\|_{L^{2}_{v},r} +  \left\|  \mathcal{P}_{N}f(t,\cdot,\cdot)- f_{N}(t,\cdot,\cdot) \right\|_{L^{2}_{v},r}\\[4pt]
	\leq & \left\|  f(t,\cdot,\cdot) - \mathcal{P}_{N}f(t,\cdot,\cdot)  \right\|_{L^{2}_{v},r} + \left\|e_{N}(t,\cdot,\cdot) \right\|_{L^{2}_{v},r}\\[4pt]
	\leq & \frac{C''_{r} \|f(t,\cdot,\cdot)\|_{H^{k}_{v},r}}{N^{k}}+ \left\|e_{N}(t,\cdot,\cdot) \right\|_{L^{2}_{v},r}.
	\end{split}
	\end{equation}
Therefore, we have
	\begin{equation}
	\frac{\rd}{\rd t} \left\| e_N(t,\cdot,\cdot) \right\|_{L^{2}_{v},r} \leq C'_{r}C_{R,L,d,2,r}^{''}(B)\left( C_{0,r}+ K_{0,r} \right) \left[ \left\|e_{N}(t,\cdot,\cdot) \right\|_{L^{2}_{v},r}+\left(\frac{C''_rK_{k,r}}{N^{k}} \right) \right],
	\end{equation}
	which implies
	\begin{equation}
	\forall t\in [0,T], \quad \left\| e_N(t,\cdot,\cdot) \right\|_{L^{2}_{v},r} \leq \e^{C'_{r}C_{R,L,d,2,r}^{''}(B)\left( C_{0,r}+ K_{0,r} \right)T}\left[\left\|e_{N}(0,\cdot,\cdot) \right\|_{L^{2}_{v},r} + \left(\frac{C''_rK_{k,r}}{N^{k}} \right) \right].
	\end{equation}
	Since $e_N(0,v,z)\equiv 0$, if we denote the constant $C_{r}(T,f^0):=\e^{C'_{r}C_{R,L,d,2,r}^{''}(B)\left( C_{0,r}+ K_{0,r} \right)T}C''_rK_{k,r}$, the desired result in \eqref{final} can be finally obtained.
\end{proof}

\section{Conclusion}
\label{sec:con}

In this paper, the Fourier-Galerkin spectral method is shown to be convergent with spectral accuracy for the homogeneous Boltzmann equation with uncertainties. We develop a brand new Sobolev space and associated norm to handle the impact of the random variables arising from the collision kernel and initial condition. In particular, we prove the well-posedness of the numerical solution obtained by the Fourier-Galerkin spectral method in our newly-established space, where the effect of the high-order regularity of the random variable is quantified as well. This paper can be regarded as an initial attempt and foundation that prepares us to study the convergence of the semi-discretized system, where the velocity and random variables are discretized simultaneously by using the Fourier-Galerkin spectral and gPC-SG method respectively. In addition, as future work, we can also study the more complicated multi-species system with random inputs and conduct the regularity and convergence analysis for numerical approximations.


\section*{Acknowledgement}
\label{sec:ack}
L.~Liu acknowledges the support by National Key R\&D Program of China (2021YFA1001200), Ministry of Science and Technology in China, Early Career Scheme (24301021) and General Research Fund (14303022 \& 14301423) funded by Research Grants Council of Hong Kong from 2021-2023. K.~Qi is supported by grants from the School of Mathematics, University of Minnesota and thanks the hospitality from the Department of Mathematics, The Chinese University of Hong Kong, where part of this work was completed.
	

\vskip2mm

\par{\bf Appendix.}\, 

\setcounter{equation}{0}
\renewcommand{\theequation}{A.\arabic{equation}}
\renewcommand{\thesubsection}{A.\arabic{subsection}}	

\subsection{Estimation of the collision operator \texorpdfstring{$Q^{R}$}{QR} for fixed random variable}\label{app_estimate_Q}

In this appendix, we recall the estimation of the collision operator $Q^{R}$ for any fixed random variable $z \in I_{z}$ without considering its high-order derivative, which is literally the classical estimation of $Q^{R}$ in the deterministic case as in \cite[Proposition 3.1]{HQY21}, except that the random variable $z \in I_{z}$ should be explicitly indicated.

\begin{proposition}\label{QR_1}
Let the truncation parameters $R$, $L$ satisfy \eqref{RL}, and assume that the collision kernel $B$ satisfy \eqref{kernel}--\eqref{cutoff} and  \eqref{kinetic}--\eqref{Assump_B}, then for $1\leq p \leq \infty$, the truncated collision operators $ Q^{R,\pm}(g,f)$ satisfy: for all $z \in I_z$, 
	\begin{equation}\label{QGLp1}
	\left\|Q^{R,+}(g,f)(\cdot, z)\right\|_{L_v^{p}(\D)} \leq C^+_{R,L,d,p}(B) \left\|g(\cdot, z)\right\|_{L_v^{1}(\D)} \left\|f(\cdot, z) \right\|_{L_v^{p}(\D)},
	\end{equation}
where the constant 
 $\displaystyle C^+_{R,L,d,p}(B)=C^{1/p}\sup_{z\in I_z}\|b(\cdot, z)\|_{L^1(\mathbb{S}^{d-1})}\|\mathbf{1}_{|v|\leq R} \Phi(|v|)\|_{L_v^{\infty}{(\D)}}$.
	\begin{equation}\label{QLLp}
	\left\|Q^{R,-}(g,f)(\cdot, z)\right\|_{L_v^{p}(\D)} \leq C_{R,L,d}^-(B) \left\|g(\cdot, z)\right\|_{L_v^{1}(\D)} \left\|f(\cdot, z)\right\|_{L_v^{p}(\D)},
	\end{equation}
	where the constant
	$\displaystyle C^-_{R,L,d}(B)=C \sup_{z\in I_z}\|b(\cdot, z)\|_{L^1(\mathbb{S}^{d-1})} \left \| \mathbf{1}_{|v|\leq R}\Phi(|v|)\right\|_{L^{\infty}(\D)}$.
In particular, for the whole collision operator $ Q^{R}(g,f)$, we have, for all $z \in I_z$,
	\begin{equation}\label{QLp}
	\left\|Q^{R}(g,f)(\cdot, z)\right\|_{L_v^{p}(\D)} \leq C_{R,L,d,p}(B) \left\|g(\cdot, z)\right\|_{L_v^{1}(\D)} \left\|f(\cdot, z) \right\|_{L_v^{p}(\D)}.
	\end{equation}	
\end{proposition}

\subsection{Local well-posedness of \texorpdfstring{$\partial_z^l f_N$}{zlfN}}\label{app_local_g}

As mentioned in Section \ref{subsec:wellposed}, the well-posedness of $f_{N}$ is obtained under our newly defined Sobolev space where each order of $z$-derivative is added up to $r$. Nevertheless, one can also establish the local existence and uniqueness results for each $z$-derivatives of $\fe$, that is, the local well-posedness of $\partial_z^l \fe$ for all $|l|\leq r$ in the $L_{v,z}^{p,\infty}$ space.

Denote $\gn(t,v,z) := \partial_z^l \fe(t,v,z)$ with $|l|\leq r$ (for notation simplicity, we omit the $l$-dependence in $\gn$). Take $l$-th order $z$-derivative on both sides of $\partial_t \fe = \mathcal{P}_N Q^R(\fe,\fe)$, then $g_{N}$ satisfies the numerical system
\begin{equation}\label{eqn-g} 
	\partial_t \gn(t,v,z) = \mathcal{P}_N \partial_z^l Q^R(\fe,\fe)(t,v,z). 
\end{equation}

\begin{proposition}\label{localexistence2}
Let the truncation parameters $R$, $L$ satisfy \eqref{RL}, and assume that the collision kernel $B$ satisfy \eqref{kernel}--\eqref{cutoff} and \eqref{kinetic}--\eqref{Assump_B}. 
Assume that the initial condition $g^{0}(v,z) = \partial^l_z f^{0}(v,z)$ belongs to $L_{v}^{1}(\D) \cap L_{v}^{2}(\D)$ for all $z \in I_z$. 
For the numerical system \eqref{eqn-g}, according to \eqref{fN_IC} assume that we evolve $\gn$ from a certain time $t_0$, where it satisfies
\begin{equation}
    \forall z \in I_z, \quad \|\gn(t_0,\cdot,\cdot)\|_{L_{v}^{1}} \leq 2 E_0^{g^0}, \qquad \|\gn(t_0,\cdot,\cdot)\|_{L_{v}^{2}} \leq  K_{0,l}.
	\end{equation}
Then, for all $z \in I_z$ there exists a local time $\tau$, such that \eqref{eqn-g} admits a unique solution $ \gn(t,\cdot,\cdot) \in L_{v}^{1} \cap L_{v}^{2} $ on $ [t_0,t_0+\tau]$.
\end{proposition}

\begin{proof}
Given $\bar E, \bar D > 0$ and small enough time $\tau >0$ to be specified later on, we define the space $\chi$ by
\begin{equation*}
\begin{split}
    \chi= \Big\lbrace \gn \in L^{\infty}\big([t_0,t_0+\tau]; L^{1}_{v}(\D) \cap L^{2}_{v}(\D)\big), \forall z \in I_z, &\sup\limits_{t\in [t_0,t_0+\tau]}\left\| \gn(t,\cdot,z) \right\|_{L^{1}_{v}} \leq \bar E,\\
 &\sup\limits_{t\in [t_0,t_0+\tau]}\left\| \gn(t,\cdot,z) \right\|_{L^{2}_{v}} \leq \bar D \Big \rbrace.
\end{split} 
\end{equation*}
Consider the following iteration on the equation \eqref{eqn-g} with $p \in \mathbb N$: 
\begin{equation}\label{g-iter}
\partial_t \gn^{(p+1)}(t, v, z)  = \mathcal{P}_N \partial_z^n Q^{R,{(p)}}(\fe, \fe)(t, v, z) , 
\end{equation}
where 
\begin{equation}
\begin{split}\label{Q-iter}
\partial_z^n Q^{R,(p)}(\fe, \fe)(t, v, z) : = & \sum_{|l|=0}^{|n|-1}\sum_{|m|=0}^{|l|} \binom{n}{l}\binom{l}{m} Q_{B^{n-l}}^R(\partial_z^m \fe, \partial_z^{l-m}\fe)(t, v, z)  \\[4pt]
& + \sum_{|m|=1}^{|n|-1}\binom{n}{m}Q^R(\partial_z^m \fe, \partial_z^{n-m}\fe) (t, v, z) \\
\end{split}
\end{equation}
using the same calculation as \eqref{Q_Leibniz}. Also, note that the first two terms in the RHS of \eqref{Q-iter} do not involve the time iteration index $p$ of the scheme. We aim to show that 
$\{ \gn^{(p)} \}_{p\in\mathbb{N}}$ is a Cauchy sequence in the space $\chi$. 

Let $\gn^{(0)}(t, v, z)=0$ and $\gn^{(p)}(t, v, z)$ all share the same initial datum as $\gn(t_0, v, z)$ for $p \in \mathbb{N}$. For all $ t\in [t_0, t_0+\tau]$ and all $z\in I_z$, the iteration system \eqref{g-iter} becomes 
$$ \gn^{(p+1)} (t, v, z) = \gn^{(p+1)}(t_0, v, z) + \int_{t_0}^t \mathcal{P}_N \partial_z^n Q^R(\fe, \fe)^{(p)}(s,v,z)\, \rd s. $$
Taking the $L_v^2$ norm on both sides, one has, for all $z \in I_z$
\small{
\begin{equation}\label{g-iter2}
\begin{split}
&\left\| \gn^{(p+1)}(t,\cdot,z) - \gn^{(p)}(t,\cdot,z) \right\|_{L^2_v} \\
\leq & \int_{t_0}^t \Big\| \left[\mathcal{P}_N \left( Q^R(\fe, \gn^{(p)}) + Q^R(\gn^{(p)}, \fe)\right)  - \mathcal{P}_N \left( Q^R(\fe, \gn^{(p-1)}) + Q^R(\gn^{(p-1)}, \fe)\right)\right](t,\cdot,z) \Big\|_{L^2_v}\, \rd s \\[4pt]
\leq & \int_{t_0}^t \underbrace{\left\| \left[Q^R(\fe, \gn^{(p)}) - Q^R(\fe, \gn^{(p-1)}) + Q^R(\gn^{(p)}, \fe) - Q^R(\gn^{(p-1)}, \fe)\right](t,\cdot,z) \right\|_{L^2_v}}_{\text{Term I}} \,\rd s . 
\end{split}
\end{equation}
}
By the estimate \eqref{QLp} in Proposition \ref{QR_1}, hence 
\begin{equation}\label{TermI}
\begin{split}
\text{Term I} & \leq \left\| Q^R(\fe, \gn^{(p)})(t,\cdot,z) - Q^R(\fe, \gn^{(p-1)})(t,\cdot,z) \right\|_{L^2_v}\\
&\qquad \qquad \qquad \qquad \qquad \qquad \qquad + \left\| Q^R(\gn^{(p)}, \fe)(t,\cdot,z) - Q^R(\gn^{(p-1)}, \fe)(t,\cdot,z)\right\|_{L^2_v} \\[4pt]
& = \left\| Q^R(\fe, \gn^{(p)}-\gn^{(p-1)})(t,\cdot,z) \right\|_{L^2_v} + \left\| Q^R(\gn^{(p)}-\gn^{(p-1)}, \fe)(t,\cdot,z) \right\|_{L^2_v} \\[4pt]
& \leq C_{R,L,d,2}(B) \Big( \left\| \fe(t,\cdot,z) \right\|_{L^1_v} \left\| \gn^{(p)}(t,\cdot,z)-\gn^{(p-1)}(t,\cdot,z) \right\|_{L^2_v}\\
&\qquad \qquad \qquad \qquad \qquad \qquad \qquad + \left\| \gn^{(p)}(t,\cdot,z)-\gn^{(p-1)}(t,\cdot,z)\right\|_{L^1_v}  \left\| \fe(t,\cdot,z) \right\|_{L^2_v}\Big)  \\[4pt]
& \leq C_{R,L,d,2}(B) \Big( \left\| \fe(t,\cdot,z) \right\|_{L^1_v} \left\| \gn^{(p)}(t,\cdot,z)-\gn^{(p-1)}(t,\cdot,z) \right\|_{L^2_v}\\
&\qquad \qquad \qquad \qquad \qquad \qquad \qquad + (2L)^{d/2} \left\| \gn^{(p)}(t,\cdot,z)-\gn^{(p-1)}(t,\cdot,z)\right\|_{L^2_v}  \left\| \fe(t,\cdot,z) \right\|_{L^2_v}\Big). 
\end{split}
\end{equation}
In Proposition \ref{localexistence}, we have already shown that there exists a unique solution $\fe(t,\cdot, z) \in L_v^1 \cap L_v^2(\mathcal{D}_L)$ on the small time interval $[t_0, t_0+\tau]$
and satisfies $\left\| \fe (t,\cdot, \cdot)\right\|_{L^1_v}\leq \bar E$ and $\left\| \fe (t,\cdot, \cdot)\right\|_{L^2_v}\leq \bar D$. Hence, \eqref{TermI} gives, for all $z \in I_z$,
\begin{equation}\label{TermI-2}
\begin{split}
\text{Term I} & \leq C_{R,L,d,2}(B) \Big( \bar E \left\| \gn^{(p)}(t,\cdot,z)-\gn^{(p-1)}(t,\cdot,z)\right\|_{L^2_v} + (2L)^{d/2}\, \bar D \left\| \gn^{(p)}(t,\cdot,z)-\gn^{(p-1)}(t,\cdot,z)\right\|_{L^2_v} \Big) \\[4pt]
& = \left( C_{R,L,d,2}(B)\, \bar E + (2L)^{d/2}\, C_{R,L,d,2}(B)\, \bar D \right) \left\| \gn^{(p)}(t,\cdot,z) - \gn^{(p-1)}(t,\cdot,z) \right\|_{L^2_v}. 
\end{split}
\end{equation}
Therefore, for $t\in [t_0, t_0+\tau]$ and all $z \in I_z$, by \eqref{g-iter2} and \eqref{TermI-2}, one has, for all $z \in I_z$,
\begin{multline*}
 \left\| \gn^{(p+1)}(t,\cdot,z)-\gn^{(p)}(t,\cdot,z)\right\|_{L^2_v} \\
 \leq 
 \tau\, C_{R,L,d,2}(B) \left(\bar E +  (2L)^{d/2}\, \bar D\right) \sup_{t\in[t_0, t_0+\tau]} \left\| \gn^{(p)}(t,\cdot,z)-\gn^{(p-1)}(t,\cdot,z)\right\|_{L^2_v}. 
\end{multline*}
By letting $\tau\, C_{R,L,d,2}(B) \left( \bar E + (2L)^{d/2}\, \bar D \right) < 1$, we obtain that $\{ \gn^{(p)} \}_{p\in \mathbb{N}}$ is a Cauchy sequence in $\chi$. Thus $\{ \gn^{(p)} \}_{p\in \mathbb{N}}$ converges to a function $\gn$ solved by \eqref{eqn-g}. 
\end{proof}


\vskip2mm



\bibliographystyle{amsplain}
\bibliography{Qi_bibtex}








\end{document}